\documentclass[11pt,a4paper]{article}

\usepackage{amsfonts, latexsym, amssymb, amsthm, amsmath, mathrsfs, esint, mathtools}
\usepackage{authblk, verbatim}

\usepackage{titlesec}
\titleformat*{\paragraph}{\itshape}
\titlespacing*{\paragraph}{0pt}{3pt plus 1pt minus 1pt}{6pt plus 2pt minus 2pt}

\usepackage[font=scriptsize]{caption}
\usepackage{graphicx}

\usepackage[a4paper, margin=2cm]{geometry}

\usepackage{color}      

\newcommand{\R}{\mathbb{R}}
\newcommand{\N}{\mathbb{N}}
\newcommand{\Z}{\mathbb{Z}}

\newcommand{\Ss}{\mathbb{S}}

\newcommand{\A}{\mathcal{A}}
\newcommand{\E}{\mathcal{E}}

\newcommand{\be}{\begin{equation}}
\newcommand{\ee}{\end{equation}}

\newcommand{\loc}{\mathrm{loc}}

\renewcommand{\div}{\mathop{\mathrm{div}}\nolimits}

\DeclareMathOperator{\supp}{supp}
\newcommand{\blank}{{\mkern 2mu\cdot\mkern 2mu}}
\DeclareMathOperator{\PV}{PV}

\newcommand{\scp}[2]{\left\langle #1, #2 \right\rangle}
\newcommand{\dd}[2]{\frac{\partial #1}{\partial #2}}
\newcommand{\set}[2]{\left\{ #1 \colon #2 \right\}}

\newtheorem{theorem}{Theorem}
\newtheorem{lemma}[theorem]{Lemma}
\newtheorem{proposition}[theorem]{Proposition}
\newtheorem{corollary}[theorem]{Corollary}
\newtheorem{conjecture}{Conjecture}
\theoremstyle{definition}
\newtheorem{definition}[theorem]{Definition}
\newtheorem{remark}[theorem]{Remark}
\newtheorem*{question}{Question}
\newtheorem*{convention}{Convention}
\theoremstyle{remark}
{}

\title{N\'eel walls with prescribed winding number and how a nonlocal term can change the energy landscape}

\author[1]{Radu Ignat}
\affil[1]{\small Institut de Math\'ematiques de Toulouse,
Universit\'e Paul Sabatier,
31062 Toulouse,
France. \authorcr
E-mail: Radu.Ignat@math.univ-toulouse.fr}
\author[2]{Roger Moser}
\affil[2]{\small Department of Mathematical Sciences,
University of Bath,
Bath BA2 7AY,
UK. \authorcr
E-mail: r.moser@bath.ac.uk}

\begin{document}

\maketitle

\abstract{We study a nonlocal Allen-Cahn type
problem for vector fields of unit length, arising from a
model for domain walls (called N\'eel walls) 
in ferromagnetism. We show that the nonlocal term
gives rise to new features in the energy landscape; in particular,
we prove existence of energy minimisers with prescribed winding number
that would be prohibited in a local model.


}


\section{Introduction} \label{sect:intro}

\subsection{Background}

We study a model for one-dimensional transition layers, called N\'eel walls, that occur in thin ferromagnetic films. In the theory of micromagnetics,
the magnetisation of a ferromagnetic sample is described by a vector
field of unit length. In a typical model for N\'eel walls, the sample
can be assumed to be two-dimensional and the vector field is tangential,
which leads to a map with values in $\Ss^1$. We use a simplified model,
also studied by several authors (see e.g. \cite{Chermisi-Muratov:13, Cote-Ignat-Miot, DKO, DKMOreduced, DKMO_rep, Ignat_Knupfer, IO, Me1, Me2}),
where it is assumed that the
transition layers have a one-dimensional profile, described by a map
$m \colon \R \to \Ss^1$.
Our model is variational and the energy functional includes the Dirichlet integral, a multi-well potential, and a nonlocal term. The geometry of the problem allows us to define a topological degree (winding number) for the magnetisation that characterises the connected components of the relevant configuration space. Therefore, it is natural to study whether these connected components contain minimisers.

The corresponding problem for an Allen-Cahn type model (without
a nonlocal term) is well understood: most connected components of the relevant space do not contain minimisers (see {A}ppendix). We will
show that the nonlocal term in our model changes the situation. In the simplest case, we will prove existence of minimisers with any prescribed winding number. We also study another case where a more intricate scenario aises: depending on a parameter, we have existence or nonexistence of
minimisers for certain winding numbers.

\subsection{The variational problem}

We now describe the energy functional studied in this paper
and the spaces where we look for minimisers. Our functional
comprises three terms, coming from four different physical
phenomena: magnetic anisotropy, an external magnetic field,
the stray field generated by the magnetisation, and the
quantum-mechanical spin interaction. The last of these gives rise
to a term called exchange energy, which is modelled simply by the
Dirichlet functional. The effects of the anisotropy and
external field have the same general structure and are combined
in effective anisotropy term in our model.

\paragraph{Anisotropy.} Fix $h \ge 0$ with $h \not= 1$ and set $k = \min\{h, 1\}\in [0,1]$.
Define an anisotropy potential $W \colon \Ss^1 \to [0, \infty)$ by
\be
\label{aniso}
W(m) = \frac{1}{2} (m_1^2 - 2hm_1 + 2hk - k^2)=
\begin{cases}
\frac 1 2 (m_1-k)^2 & \text{if } k=h<1,\\
\frac 1 2 (m_1-k)^2+(h-1)(1-m_1) & \text{if } k=1<h,
\end{cases}
\ee
for $m = (m_1, m_2) \in \Ss^1$. If $h<1$, then $W$ has two wells on $\Ss^1$, at $(k, \pm \sqrt{1-k^2})$, while
in the case $h>1$, the potential
$W$ has one well on $\Ss^1$, at $(1,0)$. In both situations, if we write $m=(\cos \theta, \sin \theta)\in \Ss^1$,  then
we have a pattern of periodically distributed wells in
terms of the phase $\theta$ and $W$ grows quadratically (in $\theta$)
near these wells (see Lemma \ref{lem:gamma}).
This behaviour is essential for our arguments and it is for this
reason why we do not study the case $h = 1$ in this paper.
In physical terms,
$W$ represents a combination of the micromagnetic anisotropy
$m\mapsto m_1^2$, with easy axis parallel to the N\'eel walls,
and an external magnetic field $h_{ext}=h {\bf e_1}$ perpendicular
to the walls.

\paragraph{Stray field potential.} Let $$\R_+^2 = \R \times (0, \infty).$$
For a given map $m=(m_1, m_2) \colon \R \to \Ss^1$ such that $m_1 - k \in H^1(\R)$,
there exists a unique solution $u \in \dot{H}^1(\R_+^2)$,
called the stray field potential, of the boundary value problem
\begin{alignat}{2}
\Delta u & = 0 & \quad & \text{in $\R_+^2$}, \label{eqn:harmonic} \\
\dd{u}{x_2} & = - m_1' && \text{on $\R \times \{0\}$}, \label{eqn:boundary_data}
\end{alignat}
where $m_1'$ denotes the derivative of $m_1$. Here $\dot{H}^1(\R_+^2)$ denotes the completion of
$C_0^\infty(\overline{\R_+^2})$ with respect to the inner
product $\scp{\blank}{\blank}_{\dot{H}^1(\R_+^2)}$, given by
\[
\scp{\phi}{\psi}_{\dot{H}^1(\R_+^2)} = \int_{\R_+^2} \nabla \phi \cdot \nabla \psi \, dx
\]
for $\phi, \psi \in C_0^\infty(\overline{\R_+^2})$. Equivalently, $u$ satisfies the identity \be
\label{weak_stray}
\int_{\R^2_+}\nabla u\cdot \nabla \zeta\, dx=\int_{-\infty}^\infty m_1'\zeta(\blank, 0)\, dx_1 \quad \text{for every } \zeta\in C^\infty_0(\R^2),
\ee
{where $x=(x_1, x_2)$}. The elements of $\dot{H}^1(\R_+^2)$ are not functions
(not even in the almost-everywhere sense), as the corresponding
norm identifies all constants. But
it is often convenient to treat them as functions nevertheless,
while keeping the ambiguity in mind. The Dirichlet integral of $u$, called the stray field energy, can be computed in terms of the homogeneous $\|\cdot\|_{\dot{H}^{1/2}}$-seminorm of $m_1$, namely \cite{Ig}
\be
\label{stray_uniq}
\frac{1}{2} \int_{\R_+^2} |\nabla u|^2 \, dx=\frac 1 2 \int_{\R} \textstyle \left|\left|\frac{d}{dx_1}\right|^{\frac{1}{2}}m_1\right|^2\, dx_1=\frac 1 2 \|m_1-k\|^2_{\dot{H}^{1/2}}.
\ee
For a discussion of how this arises from micromagnetics,
we refer to the work of DeSimone--Kohn--M\"uller--Otto \cite{DKMO_rep}.

\paragraph{Energy functional.} We now define the functional $E_h$ by the formula
\[
E_h(m) = \frac{1}{2} \int_{-\infty}^\infty \left(|m'|^2 + 2W(m)\right) \, dx_1 + \frac{1}{2} \int_{\R_+^2} |\nabla u|^2 \, dx,
\]
where $u \in \dot{H}^1(\R_+^2)$ is determined by \eqref{eqn:harmonic} and
\eqref{eqn:boundary_data}. If $h < 1$, this is well-defined and finite
for any $m \in H_\loc^1(\R; \Ss^1)$
such that $m_1 - k \in H^1(\R)$ and $m_2' \in L^2(\R)$.
If $h > 1$, then we need to assume in addition that $m_1 - 1 \in L^1(\R)$.

If $m \in H^1_\loc(\R; \Ss^1)$ with $E_h(m) < \infty$,
then it is readily seen that the limits $\lim_{x_1 \to \pm \infty} m(x_1)$
exist and coincide with one of the zeros of $W$. That is, if
$h > 1$, then
\[
\lim_{x_1 \to \pm \infty} m_1(x_1) = (1, 0),
\]
and if $h < 1$, then
\[
\lim_{x_1 \to \pm \infty} m_1(x_1) = \left(h, \pm\sqrt{1 - h^2}\right)
\]
(where the signs on both sides of the equation are independent
of one another).
We choose
\[
\alpha \in \left[0, \frac{\pi}{2}\right] \quad \text{such that } k = \cos \alpha.
\]
(Thus $\alpha = 0$ if $h > 1$.)

\paragraph{Winding number.} Let $m^\perp = (-m_2, m_1)$. It is easily seen that the quantity
\[
\deg(m) = \frac{1}{2\pi} \int_{-\infty}^\infty m^\perp \cdot m' \, dx_1
\]
exists and belongs to $\Z + \{0, \pm \frac{\alpha}{\pi}\}$ if $E_h(m) < \infty$.
Moreover, this notion of topological degree (winding number) can be extended
to all continuous maps $m \colon \R \to \Ss^1$ with
$\lim_{x_1 \to \pm \infty} m_1(x_1) = k$. More precisely, for any such continuous map $m\colon \R \to \Ss^1$, there exists a continuous function $\phi \colon \R \to \R$, called the lifting of $m$,
such that
\[
m = (\cos \phi, \sin \phi)\quad \text{in $\R$}
\]
and $\phi(\pm \infty):=\lim_{x_1 \to \pm \infty} \phi(x_1) \in 2\pi\Z+\{-\alpha, \alpha\}$. Our generalised winding number is then given by
\[
\deg(m)=\frac{\phi(+\infty)-\phi(-\infty)}{2\pi}\in \Z + \left\{0, \pm \frac{\alpha}{\pi}\right\}.
\]

\subsection{Main results}
\label{sec:main}
For any fixed $d \in \Z + \{0, \pm \frac{\alpha}{\pi}\}$,
we define
\be
\label{def:ah}
\A_h(d) = \set{m \in H_\loc^1(\R; \Ss^1)}{E_h(m) < \infty \text{ and } \deg(m) = d}
\ee
and
\[
\E_h(d) = \inf_{m \in \A_h(d)} E_h(m).
\]
Note that $\{\A_h(d)\}_{d\in \Z + \{0, \pm \frac{\alpha}{\pi}\}}$ comprises the connected components of $\{m \in H_\loc^1(\R; \Ss^1) \colon E_h(m) < \infty\}$ in the strong $\dot{H}^1(\R)$ topology. Thus it forms
a partition of this set.

The following question is studied in this paper.

\begin{question}
Given $d \in \Z + \{0, \pm \frac{\alpha}{\pi}\}$, is $\E_h(d)$
attained? That is, does $m \in \A_h(d)$ exist such that
$E_h(m) = \E_h(d)$?
\end{question}

The answer is clear for $d = 0$. Since for $m \in \A_h(d)$, we can
construct $\tilde{m}, \hat{m} \in \A_h(-d)$ by $\tilde{m}(x_1) = m(-x_1)$
and $\hat{m}_1 = m_1$, $\hat{m}_2 = -m_2$, it is also clear that
the answer will always be the same for $d$ and $-d$ (and that $\E_h(d) = \E_h(-d)$).
Therefore, it suffices to consider $d > 0$. 

In the case $h < 1$ and $d = \frac{\alpha}{\pi}$ or $d = 1 - \frac{\alpha}{\pi}$,
the answer to the question is positive and was proved
in the work of Chermisi-Muratov \cite{Chermisi-Muratov:13}
(for $h = 0$, see also the work of Melcher \cite{Me1}). In other words, 
if $h \in [0, 1)$, then $\E_h(\alpha/\pi)$ and $\E_h(1 - \alpha/\pi)$
are attained.
These papers also give a lot of information about the structure
of the minimisers.
For $h > 1$ and $d = 1$, some of their arguments still work
and give a positive answer. The underlying method relies on the
symmetrisation of $m_1$ by rearrangements and the observation that
the energy is decreased thereby. 
For higher winding numbers, the
situation is more complicated and requires different arguments.

Our first main result shows that we have energy minimisers of
any admissible winding number if $h > 1$. They correspond to
arrays of N\'eel walls as observed in physical experiments \cite[{Fig.~5.66}]{HS98}.

\begin{theorem} \label{thm:existence_h>1}
Suppose that $h > 1$. Then $\E_h(d)$ is attained for any $d \in \Z$.
\end{theorem}

In contrast, for $h < 1$, we sometimes have a negative answer.
In particular, we do not have any minimisers of winding number $1$.

\begin{theorem} \label{thm:non-existence}
If $h \in [0, 1)$, then
$\E_h(1) = \E_h(\alpha/\pi) + \E_h(1 - \alpha/\pi)$
and $\E_h(1)$ is \emph{not} attained.
\end{theorem}

In general, the case $h < 1$ is much more subtle than $h > 1$,
because the nonlocal term $\frac{1}{2} \int_{\R_+^2} |\nabla u|^2 \, dx$
in the energy gives rise simultaneously to attractive and repulsive
interactions between different parts of the profile of $m$.
We have only partial results here, but we do know the following.

\begin{theorem} \label{thm:existence_h<1}
There exists $H \in (0, 1)$ such that $\E_h(2 - \alpha/\pi)$
is attained whenever $h \in [H, 1)$.
\end{theorem}

We also prove the following Pohozaev identity for every critical point $m$ of our energy, expressing equality of
the exchange energy and the anisotropy energy.

\begin{proposition}
\label{pro:Poho}
Let $m:\R\to \Ss^1$ be a critical point of $E_h$ with $E_h(m)<\infty$. Then
$$\int_{\R} |m'|^2\, dx_1=2\int_{\R} W(m)\, dx_1.$$
\end{proposition}

We will also prove some qualitative and quantitative properties of the minimizers of $E_h$: symmetry properties (see Lemma \ref{lem:symmetry} below) and decay rates at infinity that are exponential in the case of $h>1$ and polynomial if $h<1$, respectively (see Theorems \ref{thm:exponential_decay}, \ref{thm:decay_h<1} and \ref{thm:alpha_large} below).

\subsection{Heuristics} \label{sect:heuristics}

The key to the proofs of our results is control of
the nonlocal energy. For this purpose,
we need to understand the shape of energy minimising profiles $m$. A prescribed winding number $d$ 
gives rise to a certain number of transitions of $m$ between the wells of the anisotropy potential $W$.
Each of these transitions represents a N\'eel wall (to use the
micromagnetics jargon). In the case of $h>1$, we have $2\pi$-N\'eel walls, while for $1>h=\cos \alpha$ (with $\alpha\in (0, \frac \pi 2]$), we have N\'eel walls of angle $2\alpha$ and $2\pi-2\alpha$, respectively (see Figure
\ref{fig:arrows}).

\begin{figure}[htbp]
\centering
\begin{minipage}{.3\linewidth}
\centering
\includegraphics[width=\linewidth]{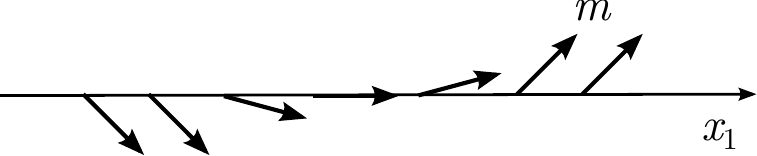}
\end{minipage}
\hspace{1cm}
\begin{minipage}{.3\linewidth}
\centering
\includegraphics[width=\linewidth]{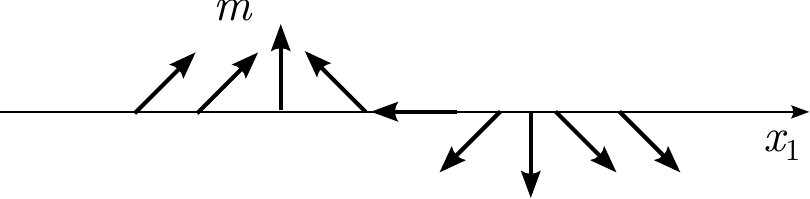}
\end{minipage}
\caption{Schematic representation of a N\'eel wall of angle $2\alpha$ (left) and $2\pi - 2\alpha$ (right).}
\label{fig:arrows}
\end{figure}

In terms of the $m_1$ component, we can distinguish these two types of walls
as follows: if $h<1$, then a wall of angle $2\alpha$ entails that $m_1$ attains the value
$1$ somewhere during the transition and we expect that $m_1$ exceeds $\cos \alpha$ throughout, while for a wall of angle $2\pi-2\alpha$,
we expect that $m_1$ is below $\cos \alpha$ and attains $-1$ at some point. For $h>1$ (i.e., when $W$ has a single well at $(1,0)$), only the second alternative can occur (see Figure \ref{fig:h>1}). 

Our first observation is that the stray field energy will give rise to attraction between pairs of walls where $m_1-\cos \alpha$ 
has the same sign, and repulsion otherwise. In particular, in the case $h>1$, we only have attraction. 
We will prove that this effect of the nonlocal energy term dominates the interaction coming from the local energy terms. 

As our energy controls the ${H}^1$-norm of $m$, the only possible cause for lack of compactness is escape to infinity of some walls. We can rule this
out, using the previously described attraction, in the following cases.
\begin{enumerate}
\item If $h>1$, only attraction is possible; this is the situation in Theorem \ref{thm:existence_h>1} (see also Figure \ref{fig:h>1}).

\item If $h<1$, the attraction between the outermost walls may
be strong enough to keep the whole profile together. This is the case in Theorem \ref{thm:existence_h<1} where a small wall is ``sandwiched'' between two large walls (see Figure \ref{fig:h<1}, right).
\end{enumerate}

\begin{figure}[htbp]
\centering
 \includegraphics[width=.3\linewidth]{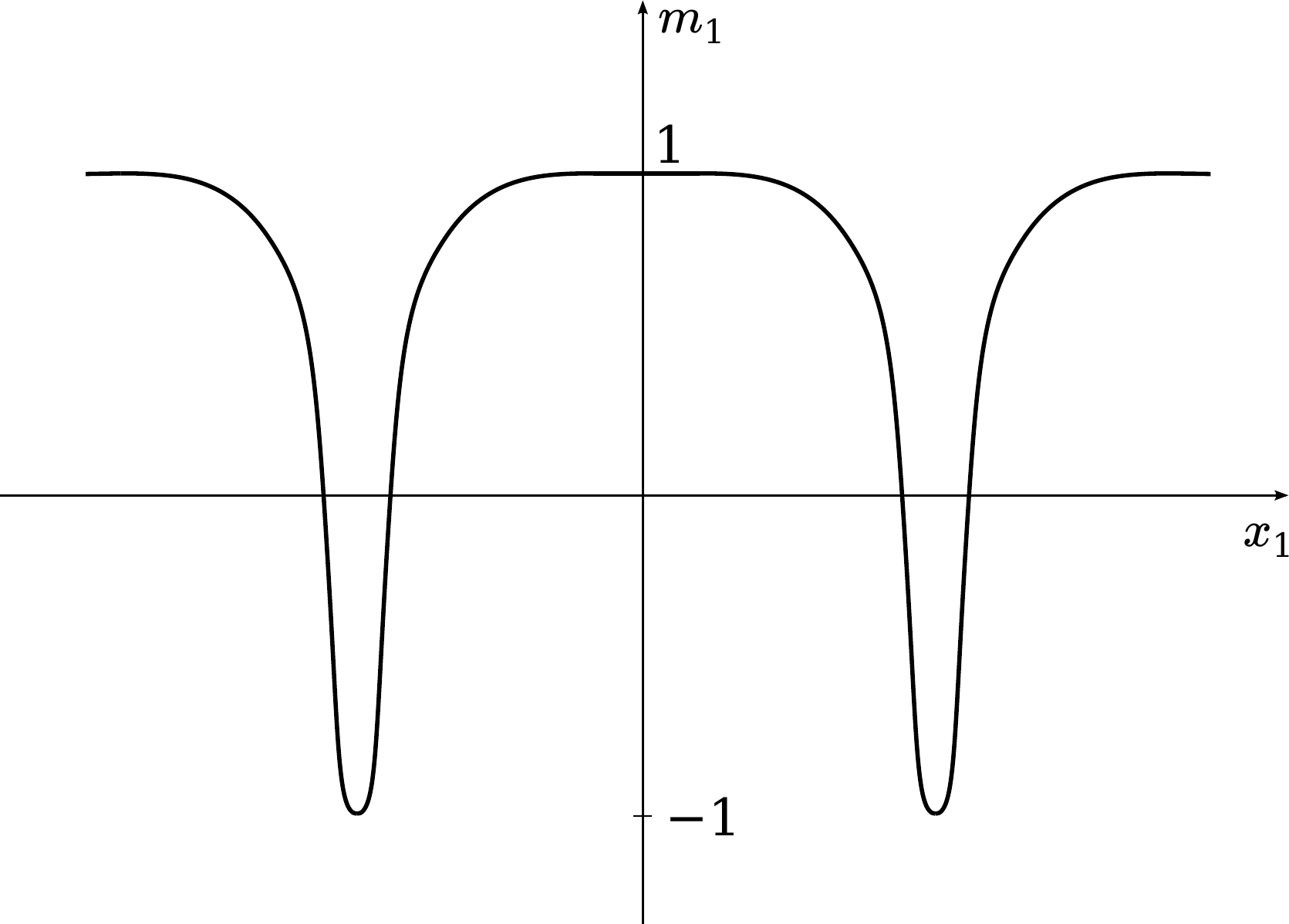}
\caption{For $h > 1$, a pair of N\'eel walls of total winding number $2$, represented in terms of $m_1$.}
\label{fig:h>1}
\end{figure}

On the other hand, if one of the outermost walls is small relative to the adjoining one (or of comparable size), then there will be a strong repulsion
that cannot be compensated by the remaining profile (as it is further away), in which case we expect nonexistence (see Figure \ref{fig:h<1}, left).
We prove this when $h<1$ and the winding number is one (see Theorem \ref{thm:non-existence}).

\begin{figure}[htbp]
\centering
 \begin{minipage}{0.3\linewidth}
 \includegraphics[width=\textwidth]{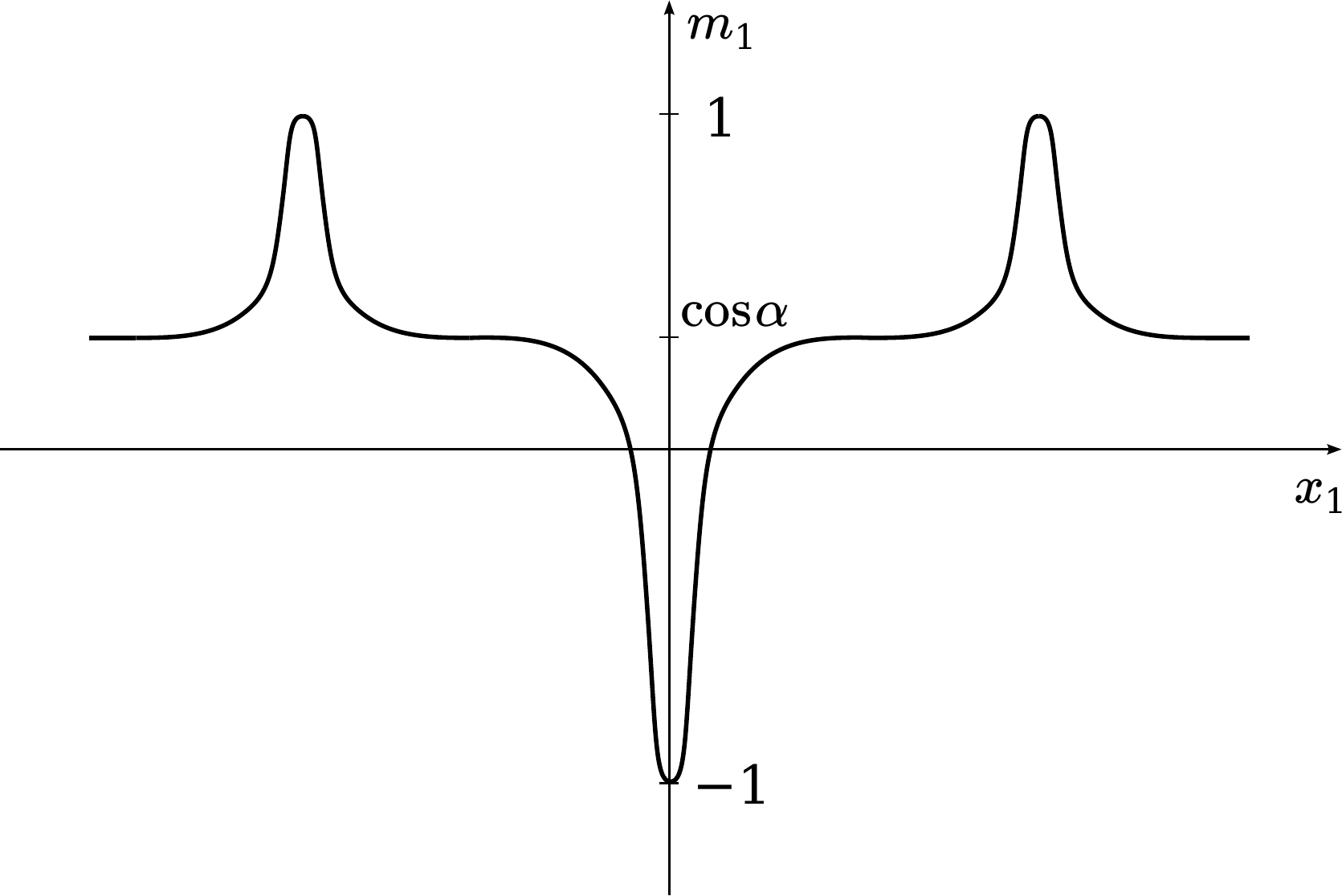}
 \end{minipage}
 \hspace{1cm}
 \begin{minipage}{0.3\linewidth}
 \includegraphics[width=\textwidth]{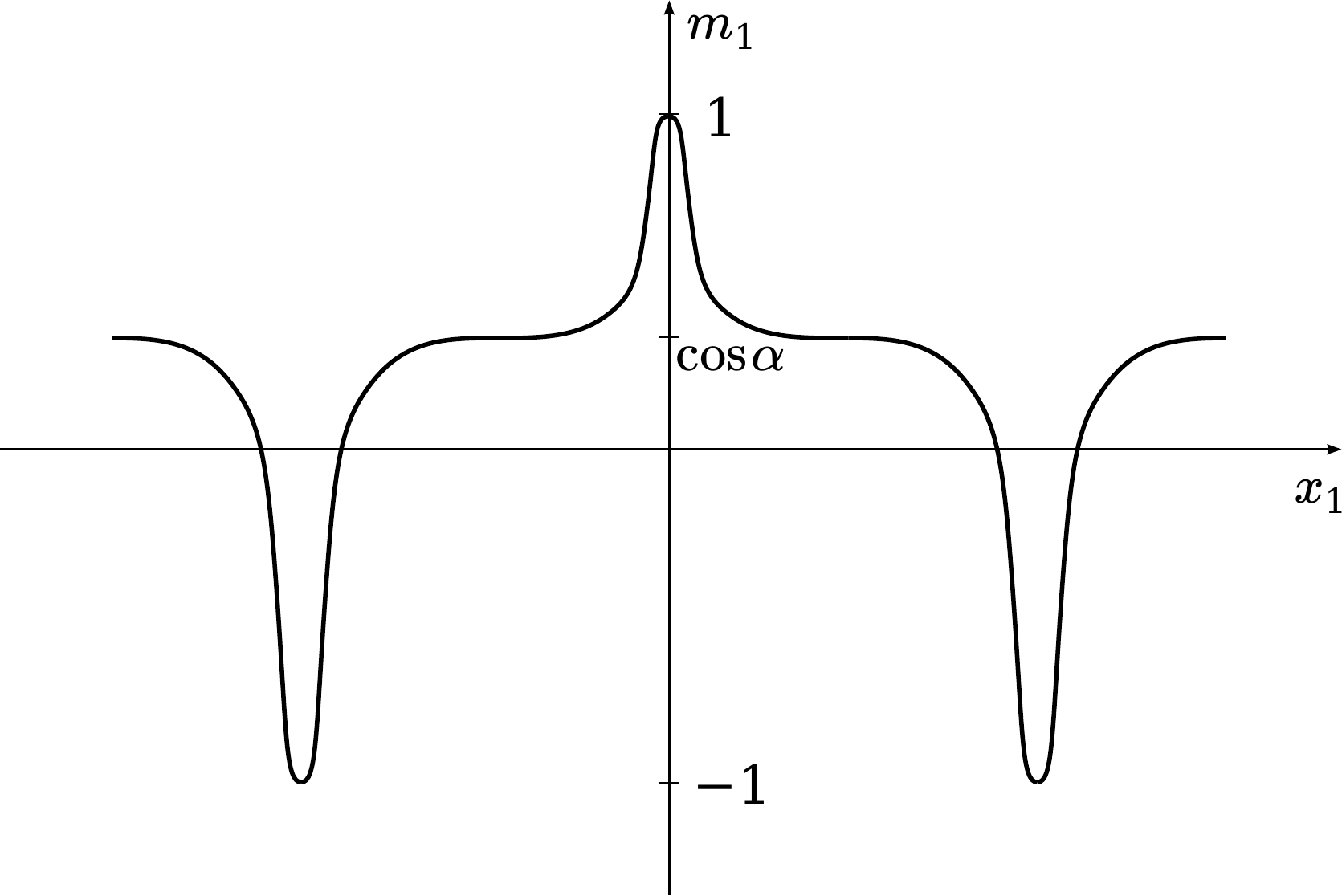}
 \end{minipage}
\caption{For $h < 1$, a hypothetical array of N\'eel walls of total winding number $1 + \alpha/\pi$ (left) and an existing one of winding number $2 - \alpha/\pi$ (right).}
  \label{fig:h<1}
\end{figure}

In the remaining cases, we do not have any proof yet, but the following behaviour
seems plausible.

\begin{conjecture}
If $h \in [0, 1)$, then for any $d \in \N=\{1,2, \dots\}$, neither $\E_h(d)$
nor $\E_h(d + \alpha/\pi)$ are attained.
\end{conjecture}

\begin{conjecture}
For any $d \in \N$, there exists $H \in (0, 1)$ such that
$\E_h(d - \alpha/\pi)$ is attained whenever $h \in [H, 1)$.
\end{conjecture}

\begin{conjecture}
For any $d \in \N \setminus \{1\}$, there exists $K \in (0, 1)$ such
that $\E_h(d - \alpha/\pi)$ is \emph{not} attained whenever $h \in [0, K]$.
\end{conjecture}

\subsection{Other representations of the energy and the winding number}

It is sometimes convenient to represent the energy functional
$E_h$ in terms of a phase (lifting) $\phi$ of $m$ such that
$m = (\cos \phi, \sin \phi)$. Abusing notation and writing
$W(\phi)$ and $E_h(\phi)$ instead of $W(m)$ and $E_h(m)$,
respectively, we have
\[
E_h(\phi) = \frac{1}{2} \int_{-\infty}^\infty \left((\phi')^2 + 2W(\phi)\right) \, dx_1 + \frac{1}{2} \int_{\R_+^2} |\nabla u|^2 \, dx.
\] 

By definition, the potential $W$ depends only on $m_1$, so
we abuse notation further and
write $W(m_1)$ instead of $W(m)$ when convenient.
Since the stray field energy is determined by $m_1$ as well,
we can rewrite the energy $E_h$ in terms of $m_1$ only:
\[
E_h(m) = \frac{1}{2} \int_{-\infty}^\infty \left(\frac{(m_1')^2}{1-m_1^2} + 2W(m_1)\right) \, dx_1 + \frac 1 2 \int_{\R} \textstyle \left|\left|\frac{d}{dx_1}\right|^{\frac 1 2}m_1\right|^2\, dx_1.
\]
In fact, often it is convenient to study our variational problem
in terms of $m_1$ only, ignoring the second component $m_2$.
Then we note that the winding number is characterised
implicitly by the following simple observation.

\begin{lemma} \label{lem:degree}
Let $d \in \N + \{0, \pm \frac{\alpha}{\pi}\} \cup \{\frac{\alpha}{\pi}\}$.
Let $m_1 \colon \R \to [-1, 1]$ be a continuous
function with $\lim_{x_1 \to \pm \infty} m_1(x_1) = k$.
Suppose that there exist $a_1, \ldots, a_I \in \R$ with
$a_1 < a_2 < \cdots < a_I$ and there exists $\epsilon \in \{\pm 1\}$
such that $m_1(a_j) = \epsilon (-1)^j$ for
$j = 1, \ldots, I$. Further suppose that one of the following conditions
is satisfied:
\begin{enumerate}
\item \label{item:degree_d-alpha/pi}
$I$ is odd and $d = \frac{I + 1}{2} - \frac{\alpha}{\pi}$ and $\epsilon = 1$; or
\item $I$ is odd and $d = \frac{I - 1}{2} + \frac{\alpha}{\pi}$
and $\epsilon = -1$ and $h < 1$; or
\item $I$ is even and $d = \frac{I}{2}$ and $h < 1$.
\end{enumerate}
Then there exists a continuous function $m_2 \colon \R \to [-1, 1]$
such that the map $m = (m_1, m_2)$ takes values in $\Ss^1$ and
$\deg(m) = d$.
\end{lemma}

\begin{proof}
We only give the arguments under the condition \ref{item:degree_d-alpha/pi},
as the proof is similar for the other cases.
Since we need to satisfy $m_1^2 + m_2^2 = 1$ everywhere,
we only need to determine the sign of $m_2$. Assuming
that \ref{item:degree_d-alpha/pi} is satisfied, we can do this
as follows: in $(-\infty, a_1)$, we choose $m_2 = \sqrt{1 - m_1^2}$;
in $[a_j, a_{j + 1})$, we choose $m_2 = (-1)^j \sqrt{1 - m_1^2}$ for
$j = 1, \ldots, I - 1$; and in $[a_I, \infty)$, we choose
$m_2 = -\sqrt{1 - m_1^2}$. This clearly gives rise to a map
$m = (m_1, m_2)$ with the desired winding number.
\end{proof}

\subsection{Notation}
\label{sec:notation}

\paragraph{The stray field potential $U(m)$.} Recalling the Neumann problem \eqref{eqn:harmonic}--\eqref{eqn:boundary_data} for $m_1-k\in H^1(\R)$,
we highlight that the solutions $u$ in $\dot{H}^{1}(\R_+^2)$ have a limit for $|x| \to \infty$.
Indeed, if we extend $u$ to $\R^2$ by even reflection, then we obtain
a harmonic function near $\infty$ with finite Dirichlet
energy, and it is well-known that the limit exists at $\infty$.
Then we normalise this constant and define $U(m)$
(sometimes also denoted $U(m_1)$) to be the unique solution
of \eqref{weak_stray} in $\dot{H}^{1}(\R_+^2)$
with\label{def:U}
\[
U(m) \to 0 \quad \text{as } |x| \to \infty.
\]
If we denote the Fourier transform with respect to $x_1$
by $\cal F$, then the solution $U(m)$ is given by \cite[Proposition 4]{Ig}
\be
\label{def:Um}
{\cal F}U(m)(\xi, x_2)=\frac{1}{\sqrt{2\pi}}\int_{\R} e^{-i\xi x_1} U(m)(x)\, dx_1= \frac{e^{-|\xi| x_2}}{|\xi|} {\cal F}(m_1'){ (\xi)}, \quad \xi\in \R, \ x_2\geq 0.\ee
Note that $U(m)\in L^2(\R^2_+)$ if, and only if, $m_1-k\in \dot{H}^{-1/2}(\R)$, where the homogeneous Sobolev space $\dot{H}^{s}(\R)$ (for $s\in \R$) is the set of tempered distributions $f$ such that ${\cal F}f\in L^1_\loc(\R)$ and
$$\|f\|^2_{\dot{H}^{s}(\R)}=\int_{\R} |\xi|^{2s} |{\cal F} f|^2\, d\xi<\infty.$$

\paragraph{The conjugate harmonic potential $V(m)$.} In addition,
we consider the conjugate harmonic function $V(m)\in \dot{H}^1(\R^2_+)$ (sometimes also denoted $V(m_1)$) with $$\nabla^\perp V(m) = - \nabla U(m) \quad \text{in }\, \, \R^2_+.$$ In other words, $V(m)$ is the unique solution of the Dirichlet problem
\begin{align}
\Delta V(m) & = 0 \quad \text{in $\R^2_+$}, \label{eqn:v} \\
V(m) & = m_1-k \quad \text{on $\R \times \{0\}$}. \label{eqn:v-dir}
\end{align}
Equivalently, $V(m)$ is the unique minimiser for the problem
\[
\int_{\R^2_+} |\nabla V(m)|^2 \, dx = \inf \set{\int_{\R^2_+} |\nabla v|^2 \, dx}{v \in \dot{H}^{1}(\R^2_+) \text{ with } \eqref{eqn:v-dir}}.
\]
It is given by the following formula,
similar to \eqref{def:Um} \cite[Proposition 3]{Ig}:
\be
\label{def:Vm}
{\cal F}V(m)(\xi, x_2)=e^{-|\xi| x_2} {\cal F}(m_1-k)(\xi), \quad \xi\in \R, \ x_2\geq 0.\ee
As ${\cal F}\left(x_1\to \frac{x_2}{x_1^2+x_2^2}\right)(\xi)=\sqrt{\frac{\pi}{2}}e^{-x_2|\xi|}$, we deduce the following Poisson formula:
$$V(m)(x)=\frac{x_2}{\pi}\int_{\R}\frac{m_1(t)-k}{(t-x_1)^2+x_2^2}\, dt, \quad x\in \R^2_+.$$

\paragraph{The Dirichlet-to-Neumann operator $\Lambda$.} Consider
the operator $\Lambda\colon \dot{H}^1(\R)\to L^2(\R)$ given by
$$\Lambda: f \mapsto \textstyle - \left(-\frac{d^2}{dx_1^2}\right)^{\frac{1}{2}}f, \quad \text{i.e.,} \quad {\cal F}(\Lambda f)(\xi)=-|\xi| {\cal F}f(\xi), \quad \xi\in \R.$$
We can represent $\Lambda$ by 
the following formula \cite[(3.1)]{DiNezza-Palatucci-Valdinoci:12}:
\begin{equation} \label{eqn:Dirichlet-to-Neumann}
\Lambda f(x_1) = \frac{1}{\pi} \PV\int_{-\infty}^\infty \frac{f(t) - f(x_1)}{(t - x_1)^2} \, dt,  \quad x_1\in \R.
\end{equation}
By \eqref{def:Um} and \eqref{def:Vm}, we obtain
\be
\label{egalite}
\Lambda (m_1-k)(x_1)=\frac{\partial}{\partial x_1} U(m)(x_1, 0)=\frac{\partial}{\partial x_2} V(m)(x_1, 0), \quad x_1\in \R. 
\ee
Therefore, this is a Dirichlet-to-Neumann operator for the
boundary value problem \eqref{eqn:v}--\eqref{eqn:v-dir}.
If $u=U(m)$, we will often write $u'$ for
the quantity $u'(x_1)=\frac{\partial}{\partial x_1} U(m)(x_1, 0)$, where $x_1\in \R$.

\begin{remark}
\label{rem:lambda_h1loc}
The Dirichlet-to-Neumann operator can be also defined on the space $\dot{H}^{1/2}(\R)$, such that
$\Lambda\colon \dot{H}^{1/2}(\R)\to \dot{H}^{-1/2}(\R)$.
Moreover, we have
$$\Lambda\left(H^1_{\loc}\cap L^\infty\cap\dot{H}^{1/2}(\R)\right) \subset L^2_{\loc}\cap \dot{H}^{-1/2}(\R).$$ 
Indeed, let $f\in H^1_\loc\cap L^\infty\cap \dot{H}^{1/2}(\R)$ and $R>0$. We want to show that $\Lambda f\in L^2(-R,R)$. To this end, we
choose a smooth cut-off function $\zeta$ with $\zeta \equiv 1$ in
$(-2R, 2R)$ and $\zeta \equiv 0$ outside of
$(-3R, 3R)$. We decompose $f = f_0 + f_1$, where $f_0 := f \zeta$. Clearly, $f_0\in H^1(\R)$ with compact support in $[-3R, 3R]$, so that $\Lambda f_0\in L^2(\R)$ and $f_1\in H^1_\loc\cap L^\infty\cap\dot{H}^{1/2}(\R)$. Since $\Lambda$ is linear, it is enough to show that $\Lambda f_1\in L^2(-R,R)$.
This follows from the estimate
\[
\begin{split}
    (\Lambda f_1, \eta)_{\dot{H}^{-1/2}, \dot{H}^{1/2}} &= -\frac{1}{2\pi }\int_\R \int_\R \frac
    {(f_1(t)-f_1(s))(\eta(t)- \eta(s))}{(t-s)^2} \ dt \ ds\\
    &= \frac{1}{\pi }\int_ {\R\setminus [-2R, 2R]} \int_{-R}^R \frac
    {f_1(t) \eta(s)}{(t-s)^2} \ dt \ ds\\
    &\leq \frac{2\sqrt{2} R}{\pi} \|f_1\|_{L^\infty(\R)} \|\eta\|_{L^2(\R)} \left(\int_{2R}^\infty \frac{ds}{(s-R)^4}\right)^{1/2} \\
    &=  \frac{2\sqrt{2}}{\pi\sqrt{3R}} \|f_1\|_{L^\infty(\R)} \|\eta\|_{L^2(\R)}
  \end{split}
\]
for $\eta \in C^\infty_0(-R,R)$.

\end{remark}

\begin{convention}
Throughout the paper, when we speak of a \emph{universal constant},
we mean a constant that depends neither on the parameter $h$
nor on any of the variables of the problem.
\end{convention}

\subsection{Organisation of the paper}

The rest of the paper is devoted to the proofs of our
results. We first prove a few auxiliary statements in
Sect.\ \ref{sect:prelim}. Among these are estimates for
$\mathcal{E}_h$, a proof that $W(\phi)$ grows quadratically
in the phase $\phi$ near its zeros, and estimates of
the energy for a profile localised with a cut-off function.

In Sect.\ \ref{sect:Euler-Lagrange}, we state the
Euler-Lagrange equation for critical points of $E_h$ and a
regularity result. We prove Proposition \ref{pro:Poho} here
and we establish further consequences of the Euler-Lagrange
equation, in particular a result on the symmetry of
minimisers and $H^2$-estimates.

As the control of the nonlocal part of the energy is
crucial for our analysis, we study this term in Sect.\ \ref{sect:nonlocal}.
We derive several estimates based on cut-off arguments similar
to Remark \ref{rem:lambda_h1loc} and we establish the
attraction/repulsion described in Sect.\ \ref{sect:heuristics}.

In Sect.\ \ref{sect:decay}, we analyse the tails of energy
minimisers and their decay as $x_1 \to \pm \infty$. For $h > 1$,
we obtain exponential decay. For $h < 1$, we can expect
polynomial decay at best, and we prove this for winding
numbers $\alpha/\pi$ and $1 - \alpha/\pi$ with the help of
a linearisation of the Euler-Lagrange equation. These
estimates are important in order to see that the attraction
or repulsion of the nonlocal terms dominates everything else.

In Sect.\ \ref{sect:concentration-compactness}, we establish
a general concentration-compactness result that allows to
prove existence of minimisers by finding good estimates
for the energy. Finally, in Sect. \ref{sect:proofs},
we combine all the ingredients and prove
Theorems \ref{thm:existence_h>1}--\ref{thm:existence_h<1}.
In order to compare our results with the situation for
a similar functional without a nonlocal term, we discuss
the known results for the latter in the {A}ppendix.

\paragraph{Acknowledgements.} Part of this research was carried out
at the ICMS Edinburgh, and the authors wish to thank the centre for
its hospitability. 
RI acknowledges partial support from the ANR project ANR-14-CE25-0009-01.

\section{Preliminary observations} \label{sect:prelim}

\subsection{A simple energy estimate}

Suppose that $h \in [0, 1)$ and we study $\E_h(\alpha/\pi)$.
While the work of Chermisi-Muratov \cite{Chermisi-Muratov:13}
gives a lot of information about this situation (especially concerning
the structure of the energy minimisers), 
we also need to know how $\E_h(\alpha/\pi)$ depends on $\alpha$
(and therefore on $h$). In particular the growth behaviour
in $\alpha$ near $0$ is important, e.g., for the proof of Theorem \ref{thm:existence_h<1}. An estimate can be obtained
by a scaling argument as follows.

\begin{lemma}[Cubic growth in $\alpha$] \label{lem:cubic_growth}
There exists a universal constant $C>0$ such that for all $h \in [0, 1)$,
\[
\E_h(\alpha/\pi) \le C\alpha^3,
\]
where $\alpha\in (0, \frac \pi 2]$ with $\cos \alpha=h$.
\end{lemma}

\begin{proof}
Choose an increasing, smooth function
$\tilde{\phi}:\R\to \R$ such that
$\lim_{x_1 \to \pm \infty} \tilde{\phi}(x_1) = \pm \pi/2$ and
$\tilde{m} = (\cos \tilde{\phi}, \sin \tilde{\phi})\in {H}^{1}(\R; \Ss^1)$. Let
$\tilde{u} = U(\tilde{m})$ as defined in \eqref{def:Um}. Note that $\tilde m\in \A_0(1/2)$ according to the notation
introduced in \eqref{def:ah}. Now define
\[
\hat{m}_1 = 1 - (1 - \cos \alpha)(1 - \tilde{m}_1).
\]
Then there exists a function $\hat{m}_2 \colon \R \to [-1, 1]$
such that $\hat{m} = (\hat{m}_1, \hat{m}_2) \in \A_h(\alpha/\pi)$.
Let $\hat{u} = U(\hat{m})$. We compute
\[
\int_{\R_+^2} |\nabla \hat{u}|^2 \, dx = (1 - \cos \alpha)^2 \int_{\R_+^2} |\nabla \tilde{u}|^2 \, dx
\]
and
\[
\int_{-\infty}^\infty (\hat{m}_1 - \cos \alpha)^2 \, dx_1 = (1 - \cos \alpha)^2 \int_{-\infty}^\infty \tilde{m}_1^2 \, dx_1.
\]
Moreover, we have
\[
1 - \hat{m}_1 = (1 - \cos \alpha) (1 - \tilde{m}_1),
\]
while
\[
1 + \hat{m}_1 \ge 1 + \tilde{m}_1.
\]
Hence
\[
\int_{-\infty}^\infty |\hat{m}'|^2 \, dx_1 = \int_{-\infty}^\infty \frac{(\hat{m}_1')^2}{1 - \hat{m}_1^2} \, dx_1 \le (1 - \cos \alpha) \int_{-\infty}^\infty \frac{(\tilde{m}_1')^2}{1 - \tilde{m}_1^2} \, dx_1.
\]
Finally, let $m(x_1) = \hat{m}\left(x_1\sqrt{1 - \cos \alpha} \right)$. Then it
follows that
\[
E_h(m) \le (1 - \cos \alpha)^{3/2} E_0(\tilde{m}),
\]
which implies the desired inequality.
\end{proof}

For the transition angle $1 - \alpha/\pi$, we have
the following uniform energy estimate.

\begin{lemma} \label{lem:energy_estimate_for_1-alpha/pi}
There exists a universal constant $C$ such that for all $h \in [0, 1)$ with $\alpha=\arccos h\in (0, \frac \pi 2]$,
\[
\E_h(1 - \alpha/\pi) \le C.
\]
\end{lemma}

\begin{proof}
Choose $\eta \in C^\infty(\R)$ with $\eta \equiv 0$ in
$(-\infty, -1]$ and $\eta \equiv 1$ in $[1, \infty)$.
Define $\phi = \alpha + (2\pi - 2\alpha) \eta$ and
$m = (\cos \phi, \sin \phi)$. Then it is clear that
$m \in \A_h(1 - \alpha/\pi)$ and
\[
\|m'\|_{L^2(\R)}=\|\phi'\|_{L^2(\R)} \le 2\pi \|\eta'\|_{L^2(\R)}.
\]
Furthermore, as $\supp (m_1 - h) \subset [-1, 1]$, we have
\[
\|m_1 - h\|_{L^2(\R)} \le 2\sqrt{2}.
\]
By standard interpolation inequalities, we obtain a
uniform estimate for $\|m_1 - h\|_{\dot{H}^{1/2}(\R)}$ as well,
and the claim follows.
\end{proof}

\subsection{Behaviour of the anisotropy $W$ near its zeros}

The function $\phi \mapsto W(\cos \phi, \sin \phi)$ grows
quadratically near its zeros. This behaviour is crucial for
our analysis and we will need the following estimates.

\begin{lemma} \label{lem:gamma}
There exists a universal constant $\gamma > 0$ such that
for all $m = (\cos \phi, \sin \phi) \in \Ss^1$ with
$\phi \in [-\pi, \pi]$, the following inequalities hold true.
If $h \in [0, 1)$ with $\alpha=\arccos h\in (0, \frac \pi 2]$, then
\[
W(m) \ge \gamma^2 (\phi^2 - \alpha^2)^2.
\]
If $h > 1$, then
\[
W(m) \ge (h - 1)(1-\cos \phi)\geq (h-1) \gamma^2 \phi^2.
\]
\end{lemma}

\begin{proof}
Suppose first that $h \in [0, 1)$.
Define the function $w \colon \R^2 \to \R$
by $w(\phi, \alpha) = \frac{1}{2}(\cos \phi - \cos \alpha)^2$
and note that $W(m) = w(\phi, \alpha)$ when $m = (\cos \phi, \sin \phi)$.
The function $w$ is smooth with vanishing derivatives up to third
order at $(0, 0)$. Moreover, we compute
\begin{gather*}
\frac{\partial^4 w}{\partial \phi^4}(0, 0) = 3, \quad
\frac{\partial^4 w}{\partial \phi^3 \partial \alpha}(0, 0) = 0, \quad
\frac{\partial^4 w}{\partial \phi^2 \partial \alpha^2}(0, 0) = -1, \\
\frac{\partial^4 w}{\partial \phi \partial \alpha^3}(0, 0) = 0, \quad
\frac{\partial^4 w}{\partial \alpha^4}(0, 0) = 3.
\end{gather*}
Therefore, by Taylor's theorem, we have
\[
\lim_{(\phi, \alpha) \to (0, 0)} \frac{w(\phi, \alpha)}{(\phi^2 - \alpha^2)^2} = \frac{1}{8}.
\]
Similarly, we see that for any $\alpha \in (0, \frac{\pi}{2}]$,
\[
\lim_{\phi \to \pm\alpha} \frac{w(\phi, \alpha)}{(\phi^2 - \alpha^2)^2} = \frac{\sin^2 \alpha}{8 \alpha^2}.
\]
This implies that the function
\[
(\phi, \alpha) \mapsto \frac{w(\phi, \alpha)}{(\phi^2 - \alpha^2)^2}
\]
has a continuous, positive extension to $[-\pi, \pi] \times [0, \frac{\pi}{2}]$.
By the compactness of this domain, the claim follows in this case.

Now suppose that $h > 1$. Then
\[
W({m}) = (h - 1) (1 - \cos \phi) + \frac{1}{2}(\cos \phi - 1)^2.
\]
As there exists a number $c > 0$ such that
$1-\cos \phi\geq c\phi^2$ for every $\phi\in [-\pi, \pi]$, the desired inequality follows in this case as well.
\end{proof}

\subsection{Localisation}

For minimisers $m$ of $E_h$, the function $m_1 - k$ will decay at a certain rate
as $x_1 \to \pm \infty$, as we will eventually see.
This will allow us to replace $m$ by a map $\tilde{m}$ such that $\tilde{m}_1 - k$
has support in a bounded interval, while changing the energy by only a
small amount. Quantifying this amount is also essential for the proof of existence of minimizers in our main results. More precisely, we have the following.

\begin{proposition} \label{prop:localisation}
There exists a constant $C>0$ with the following property. Suppose that
$\phi \in H_\loc^{1}(\R)$ is such that
$m = (\cos \phi, \sin \phi)$ satisfies $E_h(m) < \infty$. Furthermore, suppose that there
exist two numbers $\ell_\pm \in 2\pi \Z +\{-\alpha, \alpha\}$ and three measurable functions $\omega, \sigma, \tau \colon [0, \infty) \to (0, \infty)$
such that
\[
|\phi(x_1) - \ell_+| \le \omega(x_1) \quad \text{and} \quad |\phi(-x_1) - \ell_-| \le \omega(x_1) \quad \text{for all $x_1 \ge 0$}
\]
and
\[
|\phi'(x_1)| \le \sigma(|x_1|) \quad \text{and} \quad |\Lambda (m_1 - k)(x_1)| \le \tau(|x_1|) \quad \text{for all $x_1 \in \R$,}
\]
where $k=\min \{h, 1\}$. 
Let $r \ge 1$ with
\[
\sup_{x_1 \ge r} \omega(x_1) \le \begin{cases} \frac{\alpha}{2} & \text{if $h < 1$} \\ \frac{\pi}{2} & \text{if $h > 1$}. \end{cases}
\]
Then for any $R \ge r$ there exists $\tilde{m} \in H_\loc^{1}(\R; \Ss^1)$
such that
\[
\deg(\tilde{m}) = \frac{\ell_+ - \ell_-}{2\pi}, \quad \tilde{m}_1 = k \text{ in $(-\infty, -2R] \cup [2R, \infty)$},  \quad \tilde m_1=m_1 \text{ in $[-R,R]$},
\]
and 
$|\tilde{m}_1 - k| \le |m_1 - k|$
everywhere, and such that
\[
E_h(\tilde{m}) \le E_h(m) + CA \quad \text{if $h < 1$}
\]
and
\[
E_h(\tilde{m}) \le E_h(m) + C\left(\frac{1}{R} \int_R^\infty \sigma^2 \, dx_1\right)^{1/2} + CA
\quad \text{if $h > 1$,}
\]
where
\[
A = B+ \left(\int_R^\infty \omega^2 \, dx_1\right)^{1/2} B^{1/2}+ \int_R^\infty \omega \tau \, dx_1 \quad \text{and} \quad B=\int_R^\infty \left(\frac{\omega^2}{R^2} + \sigma^2\right) \, dx_1. 
\]
\end{proposition}

\begin{proof}
Choose an even function $\eta \in C^{1, 1}(\R)$
with $\eta(x_1) = 0$ for $|x_1| \ge 1$,
$\eta(x_1) = 1$ for $0 \le |x_1| \le \frac{1}{2}$,
$\eta(x_1) = (1 - |x_1|)^2$ for $\frac{3}{4} \le |x_1| < 1$,
and $\frac{1}{16} \le \eta(x_1) \le 1$ for $\frac{1}{2} < |x_1| < \frac{3}{4}$.
Fix $R \ge r$ and set $\tilde{\eta}(x_1) = \eta\left(\frac{x_1}{2R}\right)$ for every $x_1\in \R$. Now define
\[
\tilde{m}_1 = \tilde{\eta} m_1 + (1 - \tilde{\eta}) k \quad \text{in $\R$}.
\]
Then clearly $|\tilde{m}_1 - k| = \tilde{\eta} |m_1 - k| \le |m_1 - k|$.
It follows in particular that $W(\tilde{m}) \le W(m)$ pointwise in $\R$. Moreover,
since the conditions on $\omega$ prevent large oscillations of $m_1$ in
$(-\infty, -R] \cup [R, \infty)$, it is clear that there exists $\tilde{m}_2 \colon \R \to [-1, 1]$
such that the map $\tilde{m} = (\tilde{m}_1, \tilde{m}_2)$
belongs to ${H}_\loc^{1}(\R; \Ss^1)$ with $\deg(\tilde{m}) = \deg(m) = \frac{\ell_+ - \ell_-}{2\pi}$.

\paragraph{Step 1: estimate $\|\tilde{m}'\|_{L^2(\R)}$.} We compute
\[
\tilde{m}_1' = \tilde{\eta} m_1' + \tilde{\eta}' (m_1 - k) \quad \text{in $\R$}.
\]
We distinguish the cases $h < 1$ and $h > 1$. If $h < 1$, then 
\[
\frac{1}{C_1}\leq 1 - m_1^2 \le C_1(1 - \tilde{m}_1^2) \quad \text{for $|x_1|\geq R$},
\]
where $C_1>0$  is a constant that depends only on $\alpha$ (because of the condition
$\sup_{x_1 \ge r} \omega(x_1) \le \frac{\alpha}{2}$). It follows that
$$\frac{(\tilde{m}_1')^2}{1 - \tilde{m}_1^2}\leq 2\frac{\tilde \eta^2 (m_1')^2+(\tilde \eta')^2(m_1-k)^2}{1-\tilde{m}_1^2}
\leq 2C_1\frac{(m_1')^2}{1-{m}_1^2}+2C_1^2 (\tilde \eta')^2 (m_1-k)^2 \quad \text{for } |x_1|\geq R.$$
Therefore,
\[
\int_{-\infty}^\infty \frac{(\tilde{m}_1')^2}{1 - \tilde{m}_1^2} \, dx_1 \le \int_{-R}^R \frac{(m_1')^2}{1 - m_1^2} \, dx_1 + C_2 \int_R^\infty \left(\sigma^2 + \frac{\omega^2}{R^2}\right) \, dx_1
\]
for some constant $C_2 = C_2(\alpha, \eta)$.

If $h > 1$, then $1 + \tilde{m}_1\ge 1 + {m}_1 \ge 1$ in $(-\infty, -R] \cup [R, \infty)$ and
$1 - \tilde{m}_1 = \tilde{\eta} (1 - m_1)$ in $\R$. Therefore,
\[
\frac{(\tilde{m}_1')^2}{1 - \tilde{m}_1^2} \le \tilde{\eta} \frac{(m_1')^2}{1 - m_1} - 2\tilde{\eta}' m_1' + \frac{(\tilde{\eta}')^2}{\tilde{\eta}} (1 - m_1) \quad \text{for } |x_1|\geq R.
\]
Clearly, we have
\[
\int_{\R \setminus (-R, R)} \tilde{\eta} \frac{(m_1')^2}{1 - m_1} \, dx_1 \le 4 \int_R^\infty \sigma^2 \, dx_1.
\]
By the choice of $\eta$, we have $(\eta')^2/\eta \in L^\infty(\R)$. Hence there
exists a constant $C_3 = C_3(\eta)$, such that
\[
\int_{-\infty}^\infty \frac{(\tilde{\eta}')^2}{\tilde{\eta}} (1 - m_1) \, dx_1 \le \frac{C_3}{R^2} \int_{\R \setminus (-R, R)} |1 - \cos \phi| \, dx_1 \le \frac{C_3}{R^2} \int_R^\infty \omega^2 \, dx_1.
\]
Moreover,
\[
-\int_{-\infty}^\infty \tilde{\eta}' m_1' \, dx_1 \le \frac{C_4}{R} \int_R^{2R} \sigma \, dx_1 \le C_4 \left(\frac{1}{R} \int_R^\infty \sigma^2 \, dx_1\right)^{1/2}
\]
for a constant $C_4 = C_4(\eta)$. It follows that
\begin{multline*}
\int_{-\infty}^\infty \frac{(\tilde{m}_1')^2}{1 - \tilde{m}_1^2} \, dx_1 \le \int_{-R}^R \frac{({m}_1')^2}{1 - {m}_1^2} \, dx_1 + \int_R^\infty \left(4\sigma^2 + \frac{C_3 \omega^2}{R^2}\right) \, dx_1 
+ 2C_4 \left(\frac{1}{R} \int_R^\infty \sigma^2 \, dx_1\right)^{1/2}.
\end{multline*}

\paragraph{Step 2: estimate $\|\tilde{m}_1 - k\|_{\dot{H}^{1/2}(\R)}$.}
We now consider both the cases $h < 1$ and $h > 1$ together.
Note that
$\tilde{m}_1-m_1=(1 - \tilde{\eta}) (k-m_1)$, and therefore,
\[
\|m_1 - \tilde{m}_1\|_{L^2(\R)}^2 \le 2\int_R^\infty \omega^2 \, dx_1.
\]
Moreover,
\[
\|m_1' - \tilde{m}_1'\|_{L^2(\R)}^2 \le C_5 \int_R^\infty \left(\sigma^2 + \frac{\omega^2}{R^2}\right) \, dx_1
\]
for a constant $C_5 = C_5(\eta)$. By interpolation, we find that
there exists $C_6 = C_6(\eta)$ such that
\[
\|m_1 - \tilde{m}_1\|_{\dot{H}^{1/2}(\R)}^2 \le C_6 \left(\int_R^\infty \left(\sigma^2 + \frac{\omega^2}{R^2}\right) \, dx_1\right)^{1/2} \left(\int_R^\infty \omega^2 \, dx_1\right)^{1/2}.
\]
Finally, we consider $v = V(m)$ and $\tilde{v} = V(\tilde{m})$ defined by \eqref{eqn:v}--\eqref{eqn:v-dir}.
We have
\[
\begin{split}
\int_{\R_+^2} |\nabla \tilde{v}|^2 \, dx & = \int_{\R_+^2} |\nabla v|^2 \, dx + \int_{\R_+^2} |\nabla v - \nabla \tilde{v}|^2 \, dx - 2\int_{\R_+^2} \nabla v \cdot (\nabla v - \nabla \tilde{v}) \, dx.
\end{split}
\]
By the above estimate, we have
\[
\int_{\R_+^2} |\nabla v - \nabla \tilde{v}|^2 \, dx\stackrel{\eqref{stray_uniq}}{=} \|m_1 - \tilde{m}_1\|_{\dot{H}^{1/2}(\R)}^2\le C_6 \left(\int_R^\infty \left(\sigma^2 + \frac{\omega^2}{R^2}\right) \, dx_1\right)^{1/2} \left(\int_R^\infty \omega^2 \, dx_1\right)^{1/2}.
\]
An integration by parts and \eqref{egalite} yield
\[
- 2\int_{\R_+^2} \nabla v \cdot (\nabla v - \nabla \tilde{v}) \, dx = 2\int_{-\infty}^\infty (m_1 - \tilde{m}_1) \Lambda (m_1 - k) \, dx_1 \le 4\int_R^\infty \omega \tau \, dx_1.
\]
Hence
\[
\int_{\R_+^2} |\nabla \tilde{v}|^2 \, dx_1 \le \int_{\R_+^2} |\nabla v|^2 \, dx + (C_6 + 4) {(A-B)},
\]
where $A$ and $B$ are defined in the statement of the proposition.
Combining these estimates, we obtain the desired inequality for
the energy.
\end{proof}

When we apply Proposition \ref{prop:localisation}, the following estimate is useful.

\begin{lemma} \label{lem:deflections_cost_energy}
For any $c, C > 0$, there exists a number $R > 0$ such that
for any $m \in H_\loc^{1}(\R; \Ss^1)$ and any $x_1 \in \R$, the following holds true.
If $E_h(m) \le C$ and $|m_1(x_1) - k| \ge c$,
then $|m_1 - k| \ge c/2$
in $(x_1 - R, x_1 + R)$ and
\[
\int_{x_1 - R}^{x_1 + R} W(m_1) \, dx_1 \ge \frac 12\int_{x_1 - R}^{x_1 + R} (k-m_1(s))^2 \, ds\geq \frac{c^2 R}{4}.
\]
\end{lemma}

\begin{proof}
Choose $R = \frac{c^2}{16C}$. Then for every $t\in (x_1-R, x_1+R)$, we have
\[
|m_1(t)-m_1(x_1)|^2\leq {2R}\int_{x_1 - R}^{x_1 + R} |m_1'(s)|^2\, ds\leq 4RC=\frac{c^2}{4}.
\]
The conclusion is now straightforward. 
\end{proof}

As a consequence of Proposition \ref{prop:localisation}, we have the following localisation result.

\begin{corollary} \label{cor:localisation}
Let $\epsilon > 0$ and $d \in \Z + \{0, \pm \alpha/\pi\}$. Then
for any $m \in \A_h(d)$, there exist $\tilde{m} \in \A_h(d)$ and $R > 0$ such that
\[
E_h(\tilde{m}) \le E_h(m) + \epsilon
\]
and $\tilde{m}$ is constant in $(-\infty, -R]$ and in $[R, \infty)$.
\end{corollary}

\begin{proof}
It follows from Lemma \ref{lem:deflections_cost_energy} that $\lim_{x_1 \to \pm \infty} m_1(x_1) = k$.
Thus if we choose $\phi \colon \R \to \R$ with $m = (\cos \phi, \sin \phi)$,
then Proposition \ref{prop:localisation} applies with $\ell_\pm = \lim_{x_1 \to \pm \infty} \phi(x_1)$
and
\begin{align*}
\omega(x_1) & = |\phi(x_1) - \ell_+| + |\phi(-x_1) - \ell_-|, \\
\sigma(x_1) & = |\phi'(x_1)| + |\phi'(-x_1)|, \\
\tau(x_1) & = |\Lambda(m_1 - k)(x_1)| + |\Lambda(m_1 - k)(-x_1)|,
\end{align*}
provided that $r \ge 1$ is chosen sufficiently large. Since $\omega, \sigma, \tau \in L^2(0, \infty)$,
we have
\[
\lim_{R \to \infty} \int_R^\infty (\omega^2 + \sigma^2 + \tau^2) \, dx_1 = 0.
\]
Thus for a sufficiently large $R$, the inequalities of Proposition \ref{prop:localisation}
lead to the desired conclusion.
\end{proof}

\subsection{Monotonicity and subadditivity of the function $\E_h$}

In this section, we examine how the number $\E_h(d)$
depends on $d$. To this end, we construct suitable maps $m \in \A_h(d)$ and estimate their energies.

\begin{proposition}[Monotonicity] \label{prop:monotonicity}
Suppose that $d_1, d_2 \in \Z + \{0, \pm\alpha/\pi\}$
such that $0 \le d_1 \le d_2$. If $h < 1$, suppose that
$d_2 - d_1 \not= 1 - \frac{2\alpha}{\pi}$.
Then $\E_h(d_1) \le \E_h(d_2)$.
\end{proposition}

\begin{proof}
We may assume that $0 < d_1 < d_2$. Suppose that $m \in \A_h(d_2)$.
Then there exist $t_1, t_2 \in \R \cup \{\pm \infty\}$ with $t_1 < t_2$ such that\footnote{Here we use the notation
$m_1(\pm \infty) = \lim_{x_1 \to \pm\infty} m_1(x_1)$.}
$m(t_1) = (\cos \alpha, \pm \sin \alpha)$, $m(t_2) = (\cos \alpha, \pm \sin \alpha)$,
and
\[
\int_{t_1}^{t_2} m^\perp \cdot m' \, dx_1 =2\pi d_1.
\]
We then define
a map $\tilde{m} = (\tilde{m}_1, \tilde{m}_2) \colon \R \to \Ss^1$
as follows. For $x_1 \in (t_1, t_2)$, we define $\tilde{m}(x_1) = m(x_1)$.
For $x_1\not\in (t_1, t_2)$, we define $\tilde{m}_1(x_1) = m_1(x_1)$ and
$\tilde{m}_2(x_1) = \pm|m_2(x_1)|$, with the sign locally constant and chosen such that
$\tilde{m}_2$ is continuous. Then $\deg(\tilde{m}) = d_1$
and $\tilde{m} \in \A_h(d_1)$. On the other hand, we clearly have
$E_h(\tilde{m}) = E_h(m)$. Therefore, we have $\E_h(d_1) \le E_h(m)$.
The desired inequality then follows.
\end{proof}

\begin{proposition}[Subadditivity] \label{prop:subadditivity}
Suppose that $d_1, d_2, d \in \Z + \{0, \pm \alpha/\pi\}$ with $d = d_1 + d_2$.
If $\alpha = \frac{\pi}{3}$ and $d_2 - d_1 \in \Z$, suppose that $d \in \Z$.
Then
\[
\E_h(d) \le \E_h(d_1) + \E_h(d_2).
\]
\end{proposition}

\begin{proof}
Choose $m^1 \in \A_h(d_1)$ and $m^2 \in \A_h(d_2)$ and fix $\epsilon > 0$.
We want to construct $m \in \A_h(d)$ with
\begin{equation} \label{eqn:energy}
E_h(m) \le E_h(m^1) + E_h(m^2) + 3\epsilon.
\end{equation}
Using Corollary \ref{cor:localisation},
choose $R > 0$ and $\tilde{m}^1 \in \A_h(d_1)$ and
$\tilde{m}^2 \in \A_h(d_2)$ such that both are
constant in $(-\infty, -R]$ and in $[R, \infty)$ and
\[
E_h(\tilde{m}^1) \le E_h(m^1) + \epsilon \quad \text{and} \quad E_h(\tilde{m}^2) \le E_h(m^2) + \epsilon.
\]
Then there exist $\phi_1, \phi_2 : \R \to \R$ such that
$\tilde{m}^1 = (\cos \phi_1, \sin \phi_1)$ and
$\tilde{m}^2 = (\cos \phi_2, \sin \phi_2)$. Furthermore,
there exist two numbers $\beta_1, \beta_2 \in 2\pi\Z \pm \alpha$ such
that $\phi_1 = \beta_1$ in $[R, \infty)$ and $\phi_2 = \beta_2$
in $(-\infty, -R]$. Note that we can assume without loss
of generality that $\beta_2 - \beta_1 \in 2\pi \Z$.
This can be achieved by either
\begin{itemize}
\item exchanging $d_1$ and $d_2$; or
\item replacing $\phi_1(x_1)$ by $-\phi_1(-x_1)$ or $\phi_2(x_1)$ by $-\phi_2(-x_1)$; or
\item if $\alpha = \frac{\pi}{2}$, replacing $\phi_2$ by $\phi_2 + \pi$.
\end{itemize}
For $r \ge R$, define
\[
\psi(x_1) = \begin{cases}
\phi_1(x_1 + r) & \text{if $x_1 \le 0$}, \\
\phi_2(x_1 - r) - \beta_2 + \beta_1 & \text{if $x_1 > 0$}.
\end{cases}
\]
Then obviously we have
\[
\int_{-\infty}^\infty (\psi')^2 \, dx_1 = \int_{-\infty}^\infty (\phi_2')^2 \, dx_1 + \int_{-\infty}^\infty (\phi_1')^2 \, dx_1
\]
and
\[
\int_{-\infty}^\infty W(\cos \psi, \sin \psi) \, dx_1 = \int_{-\infty}^\infty W(\tilde{m}^1) \, dx_1 + \int_{-\infty}^\infty W(\tilde{m}^2) \, dx_1.
\]
Let $u^1 = U(\tilde{m}^1)$ and $u^2 = U(\tilde{m}^2)$ be defined by \eqref{def:Um}. Furthermore, let
$w = U(\cos \psi, \sin \psi)$. As \eqref{def:Um} determines $w$ uniquely,
we deduce that $w(x_1, x_2) = u^1(x_1 + r, x_2) + u^2(x_1 - r, x_2)$.
Hence
\[
\begin{split}
\int_{\R_+^2} |\nabla w|^2 \, dx & = \int_{\R_+^2} |\nabla u^1|^2 \, dx + \int_{\R_+^2} |\nabla u^2|^2 \, dx \\
& \quad + 2\int_{\R_+^2} \nabla u^1(x_1 + r, x_2) \cdot \nabla u^2(x_1 - r, x_2) \, dx.
\end{split}
\]
By Parseval's identity, the dominated convergence theorem together with the Riemann-Lebesgue lemma lead to
\begin{multline*}
\int_{\R_+^2} \nabla u^1(x_1 + r, x_2) \cdot \nabla u^2(x_1 - r, x_2) \, dx \\
= \int_0^\infty\, dx_2 \int_\R  e^{2i\xi r} {\cal F}(\nabla u^1)(\xi, x_2)\cdot 
\overline{ {\cal F}(\nabla u^2)}(\xi, x_2)\, d\xi\stackrel{r\to \infty}{\longrightarrow} 0,
\end{multline*}
where we use the fact that $ \nabla u^1,  \nabla u^2\in L^2(\R_+^2)$. Hence if $r$ is sufficiently large, the map
$m = (\cos \psi, \sin \psi)$ will satisfy \eqref{eqn:energy}.
By construction, we also have $m \in \A_h(d)$, hence this
concludes the proof.
\end{proof}

\section{The Euler-Lagrange equation} \label{sect:Euler-Lagrange}

\subsection{Statement and immediate consequences}

We now discuss critical points $m$ of the energy $E_h$. If $m \in \A_h(d)$
is a critical point of $E_h$, then it is critical relative to
$\A_h(d)$ as well, because $\A_h(d)$
is an open set in $\set{m\in H^1_\loc(\R; \Ss^1)}{E_h(m)<\infty}$
under the strong $\dot{H^1}$-topology.
Write $m = (\cos \phi, \sin \phi)\in \A_h(d)$ and let $u = U(m)$ be the function defined by \eqref{def:Um}. Then the Euler-Lagrange equation is
\begin{equation} \label{eqn:Euler-Lagrange_phi}
\phi'' = (h - \cos \phi + u') \sin \phi \quad \text{in}\quad \R.
\end{equation}
Equation \eqref{eqn:Euler-Lagrange_phi} is derived as follows:
for a test function $\zeta\in C^\infty_0(\R)$, using the notation $u = U(m) = U(m_1) = U(\cos \phi)$, we compute
\begin{align*}
\left.\frac{d}{dt}\right|_{t=0} \left(\frac 1 2  \int_{\R^2_+} |\nabla U(\cos(\phi+t\zeta))|^2\, dx\right) &=-
\int_{\R^2_+} \nabla U(m_1)\cdot \nabla U(\zeta\sin \phi)\, dx\\
&\stackrel{\eqref{weak_stray}}{=}-\int_{\R}(\zeta \sin \phi)' U(m_1) \, dx_1=\int_{\R}\zeta \sin \phi \dd{u}{x_1} \, dx_1.
\end{align*}
The other terms in \eqref{eqn:Euler-Lagrange_phi} are obtained as usual.

We can write the equation in terms of $m$, noting that
$m''=-(\phi')^2m+\phi''m^\perp$. This leads to the equation
\begin{equation} \label{eqn:Euler-Lagrange_m}
m'' + |m'|^2 m = (h - m_1 + u')m_2 m^\perp \quad \text{in } \R.
\end{equation}
Furthermore, away from $m_1^{-1}(\{\pm 1\})$, we can write the
Euler-Lagrange equation
in terms of the function $$f = m_1 - k = \cos \phi - \cos \alpha.$$
Indeed, observing that $1-m_1^2=\sin^2 \alpha - 2f \cos \alpha - f^2$, we find
the equation
\begin{equation} \label{eqn:Euler-Lagrange_m_1}
f'' = - \frac{(f')^2 (f + \cos \alpha)}{\sin^2 \alpha - 2f \cos \alpha - f^2} + (\sin^2 \alpha - 2f \cos \alpha - f^2)(f - \Lambda f - h + k) \quad \text{in } \R\setminus f^{-1}(\{\pm 1-k\}),
\end{equation}
where $\Lambda \colon \dot{H}^1(\R) \to L^2(\R)$ is the Dirichlet-to-Neumann
operator introduced in \eqref{eqn:Dirichlet-to-Neumann}.

The equation admits a regularity theory.
In particular, the following can be shown with
the arguments of Ignat-Kn\"upfer \cite[Theorem 1.1]{Ignat_Knupfer}
(even though they study a slightly different problem).
We do not give a proof here, but the main idea can also be found
in Remark \ref{rem:lambda_h1loc}.

\begin{proposition}[Regularity] \label{prop:regularity}
If $\phi \in {H}^{1}_\loc(\R)$ with $\cos \phi-k\in \dot{H}^{1/2}(\R)$ solves equation \eqref{eqn:Euler-Lagrange_phi},
then $\phi\in C^\infty(\R)$.
\end{proposition}

It is an open question whether minimisers of $E_h$
subject to a prescribed winding number (or more general, solutions
of \eqref{eqn:Euler-Lagrange_m}) necessarily correspond
to a monotone phase $\phi$. On the other hand, we can
show that a minimiser $m$ will pass through the points $(\pm 1, 0)$
exactly as many times as the winding number requires and
in a transversal way.

\begin{lemma}[Passages through $(\pm 1, 0)$] \label{lem:passages}
Suppose that $m \in \A_h(d)$ minimises $E_h$ in $\A_h(d)$. Then
$$
\begin{cases}
|m_1^{-1}(\{\pm 1\})| = 2|d| - 1 & \quad \text{if $h > 1$ and $d \in \Z \setminus \{0\}$},\\
|m_1^{-1}(\{\pm 1\})| = 2 |d| & \quad \text{if $h < 1$ and $d \in \Z$},\\
|m_1^{-1}(\{\pm 1\})| = 2 \ell - 1 & \quad \text{if $h < 1$ and $|d| = \ell-1+ \frac{\alpha}{\pi}$ or $|d| = \ell- \frac{\alpha}{\pi}$ for some $\ell \in \N$}.
\end{cases}
$$
Furthermore, if $a \in \R$ with $m_1(a) = \pm 1$, then
$m_2'(a) \not= 0$.
\end{lemma}

\begin{proof}
We may assume that $d \ge 0$. Suppose that
$\phi \colon \R \to \R$ is such that $m = (\cos \phi, \sin \phi)$. By Proposition \ref{prop:regularity}, we know that $\phi$ is smooth.

\paragraph{Step 1: prove the second statement.}
Here we show that $\phi'(a) \not= 0$ if $\phi(a) \in \pi \Z$ for some $a \in \R$. (This will then imply the second
statement of the lemma.) To this end, consider the
Euler-Lagrange equation in the form \eqref{eqn:Euler-Lagrange_phi}.
Suppose that $\phi(a) = j\pi$ with $j \in \Z$. Then
the initial value problem
\begin{align*}
\psi'' & = (h - \cos \phi + u') \sin \psi \quad \text{in $\R$}, \\
\psi(a) & = j\pi, \\
\psi'(a) & = 0,
\end{align*}
has the solution $\psi(x_1) = j\pi$. The function $\phi$
also satisfies the ordinary differential equation and
the first initial condition. But since solutions of the
initial value problem are
unique and $\phi$ cannot be constant, it follows that $\phi$
does not satisfy the second initial condition. That is,
we have $\phi'(a) \not= 0$. (This kind of argument was also used by Capella-Melcher-Otto in \cite{CMO07}.)

\paragraph{Step 2: prove the first statement.}
Now we show that $\phi(a) < \phi(b)$ for any $a, b \in \R$ with $a < b$, $\phi(a) \in \pi \Z$ and $\phi(b) \in \pi \Z$. (This will imply the first
statement of the lemma.) We argue by contradiction here.
Suppose that $\phi(a) \ge \phi(b)$. Since
\[
\lim_{x_1 \to \infty} \phi(x_1) \ge \lim_{x_1 \to -\infty} \phi(x_1)
\]
and there can be no local extrema at $a$ or $b$ by the first part
of the proof, it follows that
there exist $a', b' \in \R$ with $a' < b'$ such that
$\phi(a') = \phi(b')\in \pi \Z$. Now define
\[
\tilde{\phi}(x_1) = \begin{cases}
\phi(x_1) & \text{if $x_1 \le a'$ or $x_1 \ge b'$}, \\
2\phi(a') - \phi(x_1) & \text{if $a' < x_1 < b'$},
\end{cases}
\]
and $\tilde{m} = (\cos \tilde{\phi}, \sin \tilde{\phi})$.
Then $\tilde{m} \in \A_h(d)$ and $\tilde{m}_1 = m_1$.
Therefore, we have $E_h(\tilde{m}) = E_h(m)$ and $\tilde{m}$
is another minimiser of $E_h$ in $\A_h(d)$. Proposition \ref{prop:regularity}
implies that $\tilde{\phi}$ is smooth. Since we already know
that $\phi'(a') \not= 0$, this is impossible. Therefore,
we have in fact $\phi(a) < \phi(b)$.
\end{proof}

\subsection{Pohozaev identity}

Next we prove the Pohozaev identity from Proposition \ref{pro:Poho},
which gives equipartition between the exchange and the anisotropy energy
for critical points of $E_h$.

Before we give the rigorous proof, however, we describe
the central idea informally. For $t > 0$, let $m^t(x_1) = m(tx_1)$ for every $x_1\in \R$. We compute $\left.\frac{d}{dt}\right|_{t=1} m^t(x_1) = x_1 m'(x_1)$ and
\[
E_h(m^t) = \frac{1}{2} \int_{-\infty}^\infty \left(t|m'|^2 + \frac{2}{t}W(m)\right) \, dx_1 + \frac12\int_{-\infty}^\infty \textstyle \left|\left|\frac{d}{dx_1}\right|^{\frac 1 2}m_1\right|^2\, dx_1,
\]
noting that the $\dot{H}^{1/2}$-seminorm is invariant under scaling in $\R$.
If $m$ is a critical point of $E_h$, we expect that 
\[
0 = \left.\frac{d}{dt}\right|_{t = 1} E_h(m^t) = \frac{1}{2} \int_{-\infty}^\infty \left(|m'|^2 - 2W(m)\right) \, dx_1.
\]
For energy minimisers, the formula from Proposition \ref{pro:Poho}
follows in fact immediately. For solutions of the Euler-Lagrange
equation, however, we need additional arguments.

\begin{proof}[Proof of Proposition \ref{pro:Poho}]
We write $m = (\cos \phi, \sin \phi)$. Let $u = U(m)$ be the function defined in \eqref{def:Um}. By Proposition~\ref{prop:regularity}, we know that $\phi$ is smooth in $\R$.

We now use an argument similar to a proof in our previous paper \cite[Lemma 12]{IgnMosARMA}.
As $u$ is harmonic, we calculate, for every $R>0$, that
\[
\div \left(\frac{1}{2} |\nabla u|^2 x - (x \cdot \nabla u) \nabla u\right) = 0 \quad \text{in $B_R^+=\set{x\in \R^2}{|x|<R, x_2>0}$.}
\]
Denote $C_R^+=\set{x\in \partial B_R^+}{x_2>0}$ and $\partial_r u=\frac{x}{|x|}\cdot \nabla u$. The Gauss theorem gives
\[
\int_{C^+_R} \left(\frac{R}{2} |\nabla u|^2 - R \left(\partial_r u\right)^2\right) \, d\sigma + \int_{-R}^{R} x_1 \dd{u}{x_1} \dd{u}{x_2} \, dx_1 
=0, \quad R>0.
\]
Then \eqref{eqn:boundary_data}, \eqref{eqn:Euler-Lagrange_phi}, and an
integration by parts yield
\begin{align*}
\int_{-R}^{R} x_1 \dd{u}{x_1} \dd{u}{x_2} \, dx_1  &= \int_{-R}^{R} x_1 \big(\phi''-\partial_\phi W(\phi)\big) \phi' \, dx_1 \\
&=\frac12\bigg[x_1 \big( (\phi')^2-2W(\phi)\big)\bigg]_{-R}^R -\frac{1}{2} \int_{-R}^{R}  \big( (\phi')^2-2W(\phi)\big) \, dx_1.
\end{align*}
As $E_h(m)<\infty$ we deduce that the function
\[
R\mapsto \left(\left(\phi'(R)\right)^2+2W(\phi(R))+\left(\phi'(-R)\right)^2+2W(\phi(-R))+\int_{C^+_R} |\nabla u|^2\, d\sigma\right)
\]
belongs to $L^1(\R_+)$. Therefore,
there exists a sequence $R_k \to \infty$ such that
\[
R_k \left(\left(\phi'(R_k)\right)^2+2W(\phi(R_k))+\left(\phi'(-R_k)\right)^2+2W(\phi(-R_k))+\int_{C^+_{R_k}} |\nabla u|^2\, d\sigma\right) \to 0, \quad k\to \infty.
\]
In particular,
\[
\bigg[x_1 \big( (\phi')^2-2W(\phi)\big)\bigg]_{-R_k}^{R_k}, \quad \int_{C^+_{R_k}} \left(\frac{R_k}{2} |\nabla u|^2 - R_k \left(\partial_r u\right)^2\right) \, d\sigma\to 0, \quad k\to \infty.
\]
The dominated convergence theorem implies that
$$\frac{1}{2} \int_{-R_k}^{R_k}  \big( (\phi')^2-2W(\phi)\big) \, dx_1 \to \frac{1}{2} \int_{\R}  \big( (\phi')^2-2W(\phi)\big) \, dx_1, \quad k\to \infty.$$
The conclusion is now straightforward.
\end{proof}

\subsection{Symmetry}

As mentioned previously,  if $h < 1$ and $d = \frac{\alpha}{\pi}$
or $d = 1 - \frac{\alpha}{\pi}$, or if $h > 1$ and $d = 1$,
then symmetrisation arguments are crucial for the construction
of energy minimisers in $\A_h(d)$. Although the same arguments do
not work for higher winding numbers, there is still some symmetry.

\begin{definition}
We say that a map $m \colon \R \to \Ss^1$ is \emph{symmetric}
if $m_1$ is an even function and $m_2$ is an odd function.
\end{definition}

We prove that such symmetry holds true for minimisers of $E_h$ in $\A_h(d)$
with the exception of the case $h<1$ and $d\in \Z$. 

\begin{lemma}[Symmetry] \label{lem:symmetry}
Suppose that $d \in \Z \pm \alpha/\pi$ and $m \in \A_h(d)$.
Then there exists a symmetric map $m^* \in \A_h(d)$ with
$E_h(m^*) \le E_h(m)$. Furthermore, if $m \in \A_h(d)$ is a
minimiser of $E_h$ in $\A_h(d)$, then there exists $t_0 \in \R$
such that $m(\blank - t_0)$ is symmetric.
\end{lemma}

\begin{proof}
Without loss of generality, we may assume that $m_1(0) = 1$
if $d \in 2\Z \pm \alpha/\pi$ and $m_1(0) = -1$ if
$d \in 2\Z + 1 \pm \alpha/\pi$ and that
\[
\int_{-\infty}^0 m^\perp \cdot m' \, dx_1 = \int_0^\infty m^\perp \cdot m' \, dx_1 = \pi d.
\]
Define $m^+ = (m_1^+, m_2^+) \in \A_h(d)$ and $m^- = (m_1^-, m_2^-) \in \A_h(d)$ as follows:
\begin{align*}
m_1^+(x_1) & = \begin{cases} m_1(x_1) & \text{if $x_1 \ge 0$}, \\ m_1(-x_1) & \text{if $x_1 < 0$}, \end{cases} \\
m_2^+(x_1) & = \begin{cases} m_2(x_1) & \text{if $x_1 \ge 0$}, \\ -m_2(-x_1) & \text{if $x_1 < 0$}, \end{cases} \\
m_1^-(x_1) & = \begin{cases} m_1(-x_1) & \text{if $x_1 \ge 0$}, \\ m_1(x_1) & \text{if $x_1 < 0$}, \end{cases} \\
m_2^-(x_1) & = \begin{cases} -m_2(-x_1) & \text{if $x_1 \ge 0$}, \\ m_2(x_1) & \text{if $x_1 < 0$}. \end{cases}
\end{align*}
Define $v = V(m)$ and $v^\pm = V(m^\pm)$
as in \eqref{eqn:v}--\eqref{eqn:v-dir}. Then
$\Delta v^+ = 0$ in $\{x_2>0\}$ (in particular, $v^+$ is smooth in $\{x_2>0\}$), and by the symmetry of $m^+$, the function
$v^+(\blank, x_2)$ is even, so that
$\dd{v^+}{x_1}(0, x_2) = 0$ for every $x_2 > 0$. Of course,
we also have $v^+(x_1, 0) = m_1(x_1) - k$ for $x_1 > 0$.
It follows that the restriction of $v^+$ to $(0, \infty)^2$ is the
unique minimiser of the Dirichlet energy in $(0, \infty)^2$ subject
to these boundary data on $(0, \infty) \times \{0\}$ and free
boundary data on $\{0\} \times (0, \infty)$. In particular,
\[
\int_{(0, \infty)^2} |\nabla v^+|^2 \, dx \le \int_{(0, \infty)^2} |\nabla v|^2 \, dx,
\]
with equality if, and only if, $v = v^+$.
Similarly,
\[
\int_{(-\infty, 0) \times (0, \infty)} |\nabla v^-|^2 \, dx \le \int_{(- \infty, 0) \times (0, \infty)} |\nabla v|^2 \, dx,
\]
with equality if, and only if, $v = v^-$. Therefore, by the symmetry of
$v^\pm$, we have
\[
\frac{1}{2} \int_{\R_+^2} \left(|\nabla v^+|^2 + |\nabla v^-|^2\right) \, dx \le \int_{\R_+^2} |\nabla v|^2 \, dx,
\]
with equality if, and only if, $v = v_+ = v_-$ (which, in particular,
would mean that $m_1$ is even).
It is clear from the construction that
$$
\frac{1}{2} \int_{-\infty}^\infty \left(|(m^+)'|^2 + |(m^-)'|^2 + 2W(m^+) + 2W(m^-)\right) \, dx_1= \int_{-\infty}^\infty \left(|m'|^2 + 2W(m)\right) \, dx_1.
$$
Thus we have
\[
\frac{1}{2} \left(E_h(m^+) + E_h(m^-)\right) \le E_h(m),
\]
with equality if, and only if, $m_1$ is even.
So either $m^+$ or $m^-$ has the required properties for the
first statement.

If $m$ is an energy minimiser, then it follows immediately
that $m_1$ is even.  By Lemma \ref{lem:passages}, there exist
exactly as many points in $m_1^{-1}(\{\pm 1\})$ as required
by the winding number. Therefore, the function $m_2$ is determined
uniquely by $m_1$ and the winding number, and it follows that $m_2$ is
odd. So $m$ is symmetric.
\end{proof}

\subsection{$H^2$-estimates based on the Euler-Lagrange equation}

In this section we use the Euler-Lagrange equation
\eqref{eqn:Euler-Lagrange_phi} to derive some $H^2$-estimates for minimisers $m$ of $E_h$ in $\A_h(d)$.
Recall that by Lemma \ref{lem:passages}, such a minimiser $m$ passes
through the points $(\pm 1,0)$ a finite number of times, which means,
in particular, that $m_2\not= 0$ on an interval of the form $(a, \infty)$.
We prove the following estimate for critical points $m$ of $E_h$
under the assumption that the second component $m_2$
does not vanish on $(a, \infty)$.

\begin{lemma} \label{lem:higher_derivatives}
There exists a universal constant $C$ such that for
any solution $\phi \in C^\infty(\R)$
of \eqref{eqn:Euler-Lagrange_phi} with $u = U(\cos \phi)$,
if there exists a number
$a \in \R$ such that $\sin \phi \not= 0$ in $(a, \infty)$, then
$$
\int_{a + R}^\infty \left((\phi'')^2 + (\phi')^2 \sin^2 \phi + (\phi')^4 (1 + \cot^2 \phi) \right) \, dx_1 
+ \int_{(a + R, \infty) \times (0, \infty)} |\nabla^2 u|^2 \, dx \le \frac{C E_h(m)}{R^2}
$$
for any $R > 0$.
\end{lemma}

\begin{proof} The following arguments rely on ideas from
our previous paper \cite[Lemma 11]{IgnMosARMA}.
We first note that
\[
\frac{\phi''}{\sin \phi} = h - \cos \phi + u' \quad \text{in $(a, \infty)$}
\]
by \eqref{eqn:Euler-Lagrange_phi}. Differentiating, we obtain
\[
\frac{\phi'''}{\sin \phi} = \frac{\phi'' \phi' \cos \phi}{\sin^2 \phi} + \phi' \sin \phi + u'',
\]
and hence
\be
\label{eq:1}
\phi''' = \phi'' \phi' \cot \phi + \phi' \sin^2 \phi + u'' \sin \phi \quad \text{in $(a, \infty)$}.
\ee
Let $\eta \in C_0^\infty(\R^2)$ with $\eta (x_1, 0) = 0$ for $x_1 \not\in (a, \infty)$.
Let $v = V(\cos \phi, \sin \phi)$ be defined as in \eqref{eqn:v}--\eqref{eqn:v-dir}. Then $u''(x_1, 0) = \dd{v'}{x_2}(x_1, 0)$
and $v'(x_1,0) = - \phi'(x_1) \sin \phi(x_1)$. Multiplying \eqref{eq:1} by $\eta^2(\blank, 0) \phi'$ and integrating by parts, we obtain
\[
\begin{split}
\int_a^\infty \eta^2 (\phi'')^2 \, dx_1 & = -\int_a^\infty \eta^2 \bigl(\; \underbrace{\phi'' (\phi')^2}_{\mathclap{=\frac13 [(\phi')^3]'}} \cot \phi + (\phi')^2 \sin^2 \phi + u'' \underbrace{\phi' \sin \phi}_{\mathclap{=-(\cos \phi)'}} \; \bigr) \, dx_1 - 2 \int_a^\infty \eta \eta' \phi'' \phi' \, dx_1 \\
& = - \int_a^\infty \eta^2 (\phi')^2 \sin^2 \phi \, dx_1 - \frac{1}{3} \int_a^\infty \eta^2 (\phi')^4 (1 + \cot^2 \phi) \, dx_1 \\
& \quad + 2\int_a^\infty \eta \eta' \left(\frac{1}{3} (\phi')^3 \cot \phi - \phi'' \phi'\right) \, dx_1 - \int_{\R_+^2} \eta^2 |\nabla v'|^2 \, dx \\
& \quad - 2 \int_{\R_+^2} \eta v' \nabla \eta \cdot \nabla v' \, dx.
\end{split}
\]
We estimate
\[
-2\int_a^\infty \eta \eta' \phi'' \phi' \, dx_1 \le \frac{1}{2} \int_a^\infty \eta^2 (\phi'')^2 \, dx_1 + 2\int_a^\infty (\eta')^2 (\phi')^2 \, dx_1
\]
and
\[
\frac{2}{3} \int_a^\infty \eta \eta' (\phi')^3 \cot \phi \, dx_1 \le \frac{1}{6} \int_a^\infty \eta^2 (\phi')^4 \cot^2 \phi \, dx_1 + \frac{2}{3} \int_a^\infty (\eta')^2 (\phi')^2 \, dx_1.
\]
Furthermore,
\[
- 2 \int_{\R_+^2} \eta v' \nabla \eta \cdot \nabla v' \, dx \le \frac{1}{2} \int_{\R_+^2} \eta^2 |\nabla v'|^2 \, dx + 2\int_{\R_+^2} |\nabla \eta|^2 |\nabla u|^2 \, dx.
\]
As $u$ is harmonic and $\nabla u=-\nabla^\perp v$, we have
$\frac{\partial^2u}{\partial x_2^2}=-\frac{\partial^2u}{\partial x_1^2}=-\frac{\partial v'}{\partial x_2}$. Therefore,
the Hessian satisfies the following identity: 
$$|\nabla^2 u|^2= 2|\nabla v'|^2 \quad \text{in } \R^2_+.$$
Hence it follows that
\begin{multline*}
\int_a^\infty \eta^2 \left((\phi'')^2 + 2(\phi')^2 \sin^2 \phi + \frac{1}{3} (\phi')^4 (1 + \cot^2 \phi)\right) \, dx_1 + \frac{1}{2} \int_{\R_+^2} \eta^2 |\nabla^2 u|^2 \, dx \\
\le \frac{16}{3} \int_a^\infty (\eta')^2 (\phi')^2 \, dx_1 + 4\int_{\R_+^2} |\nabla \eta|^2 |\nabla u|^2 \, dx.
\end{multline*}
A suitable choice of $\eta$ now gives the desired inequality.
\end{proof}

\section{The nonlocal terms} \label{sect:nonlocal}

\subsection{Some estimates in $\dot{H}^{1/2}$}

Here we derive some inequalities that we will use
to estimate the stray field energy
$\frac{1}{2} \int_{\R_+^2} |\nabla u|^2 \, dx$ appearing
in $\E_h$. This  part of the
energy is the most difficult to control and is chiefly
responsible for the interesting pattern of existence and
nonexistence of minimisers described in Sect.~\ref{sect:intro}.

As we have seen, we can write
\[
\int_{\R_+^2} |\nabla u|^2 \, dx = \|m_1 - k\|_{\dot{H}^{1/2}(\R)}^2
\]
if $u \in \dot{H}^1(\R_+^2)$ is the unique solution of
\eqref{eqn:harmonic}--\eqref{eqn:boundary_data}. Therefore, the
subsequent analysis is also about the space $\dot{H}^{1/2}(\R)$ and
its inner product $\scp{\blank}{\blank}_{\dot{H}^{1/2}(\R)}$,
which can be expressed either through harmonic extensions to
$\R_+^2$ or by \cite[Theorem 7.12]{Lieb-Loss:01}:

\begin{align} \label{eqn:H^{1/2}-inner_product}
\scp{f}{g}_{\dot{H}^{1/2}(\R)} &=-\int_{\R} \Lambda f \, g\, dx_1 =\frac{1}{2\pi} \int_{-\infty}^\infty \int_{-\infty}^\infty \frac{(f(s) - f(t)) (g(s) - g(t))}{(s - t)^2} \, ds \, dt\\
\nonumber 
&{\stackrel{\eqref{def:Vm}, \eqref{egalite}}{=}\int_{\R_+^2} \nabla V(f)\cdot \nabla V(g) \, dx=\int_{\R_+^2} \nabla U(f)\cdot \nabla U(g) \, dx},
\end{align}
for $f, g \in \dot{H}^{1/2}(\R)$.
From this formula we obtain some inequalities in particular
if $f$ and $g$ have disjoint or almost disjoint supports.

\begin{lemma}[Repulsion between positive and negative parts] \label{lem:repulsion}
Let $f \in \dot{H}^{1/2}(\R)$ and define $f_+ = \max\{f, 0\}\geq 0$ and $f_- = \min\{f, 0\}\leq 0$.
Then
\[
\|f\|_{\dot{H}^{1/2}(\R)}^2 \ge \|f_+\|_{\dot{H}^{1/2}(\R)}^2 + \|f_-\|_{\dot{H}^{1/2}(\R)}^2,
\]
with equality if, and only if, $f$ does not change sign (i.e., either $f_+ = 0$ or $f_- = 0$).
\end{lemma}

\begin{proof}
By the bilinearity, this statement is equivalent to
\[
\scp{f_+}{f_-}_{\dot{H}^{1/2}(\R)} \ge 0,
\]
with equality if, and only if, either $f_+ = 0$ or $f_- = 0$.
Using \eqref{eqn:H^{1/2}-inner_product} and the fact that $f_+ f_- = 0$
in $\R$, we obtain
\[
\scp{f_+}{f_-}_{\dot{H}^{1/2}(\R)} = -\frac{1}{\pi} \int_{-\infty}^\infty \int_{-\infty}^\infty \frac{f_+(s) f_-(t)}{(s - t)^2} \, ds \, dt.
\]
It is clear that the right-hand side has the required
properties.
\end{proof}

The following inequalities are based on similar ideas.

\begin{lemma} \label{lem:separated_supports}
Suppose that $f, g \in L^2(\R) \cap \dot{H}^{1/2}(\R)$
and there exist $a \in \R$ and $R > 0$ such that
$\supp f \subset (-\infty, a - R]$ and $\supp g \subset [a + R, \infty)$.
Then
\[
\left|\langle f, g\rangle_{\dot{H}^{1/2}(\R)}\right| \le \frac{\|f\|_{L^2(\R)} \|g\|_{L^2(\R)}}{2\pi R \sqrt{6}}.
\]
\end{lemma}

\begin{proof}
We may assume that $a = 0$. We have
\[
\begin{split}
\left|\langle f, g \rangle_{\dot{H}^{1/2}(\R)}\right| & = \frac{1}{2\pi} \left|\int_{-\infty}^\infty \int_{-\infty}^\infty \frac{(f(s) - f(t))(g(s) - g(t))}{(s - t)^2} \, ds \, dt\right| \\
& \le \frac{1}{\pi} \int_{-\infty}^\infty \int_{-\infty}^\infty \frac{|f(s)| |g(t)|}{(s - t)^2} \, ds \, dt.
\end{split}
\]
For any $t \ge R$,
\[
\int_{-\infty}^\infty \frac{|f(s)|}{(s - t)^2} \, ds \le \|f\|_{L^2(\R)} \left(\int_{-\infty}^{-R} \frac{ds}{(s - t)^4}\right)^{1/2} = \frac{\|f\|_{L^2(\R)}}{\sqrt{3(t + R)^3}}.
\]
Thus
\[
\begin{split}
\int_{-\infty}^\infty \int_{-\infty}^\infty \frac{|f(s)| |g(t)|}{(s - t)^2} \, ds \, dt & \le \|f\|_{L^2(\R)} \int_R^\infty \frac{|g(t)|}{\sqrt{3(t + R)^3}} \, dt \\
& \le \|f\|_{L^2(\R)} \|g\|_{L^2(\R)} \left(\int_R^\infty \frac{dt}{3(t + R)^3}\right)^{1/2} \\
& = \frac{\|f\|_{L^2(\R)} \|g\|_{L^2(\R)}}{2R\sqrt{6}}.
\end{split}
\]
The claim now follows.
\end{proof}

\begin{lemma} \label{lem:attraction}
Suppose that $f, g \in \dot{H}^{1/2}(\R)$ are nonnegative functions and $R > 0$ with
$\supp f \subset [-2R, -R]$ and $\supp g \subset [R, 2R]$. Then
\[
- \frac{1}{4\pi R^2} \|f\|_{L^1(\R)} \|g\|_{L^1(\R)} \le \langle f, g\rangle_{\dot{H}^{1/2}(\R)} \le - \frac{1}{16 \pi R^2} \|f\|_{L^1(\R)} \|g\|_{L^1(\R)}.
\]
\end{lemma}

\begin{proof}
Again we have
\[
\langle f, g\rangle_{\dot{H}^{1/2}(\R)} = - \frac{1}{\pi} \int_{-\infty}^\infty \int_{-\infty}^\infty \frac{f(s) g(t)}{(s - t)^2} \, ds \, dt.
\]
But $t - s \le 4R$ for $t \in \supp g$ and $s \in \supp f$. Hence
\[
\langle f, g\rangle_{\dot{H}^{1/2}(\R)} \le -\frac{1}{16 \pi R^2} \int_{-\infty}^\infty \int_{-\infty}^\infty f(s) g(t) \, ds \, dt = - \frac{1}{ 16 \pi R^2} \|f\|_{L^1(\R)} \|g\|_{L^1(\R)}.
\]
The other inequality follows similarly.
\end{proof}

\subsection{Pointwise estimates for the Dirichlet-to-Neumann operator}

When analysing the Euler-Lagrange equation for minimisers of $E_h$, we need
to control in particular the term involving the non-local Dirichlet-to-Neumann operator $\Lambda$ defined by
\eqref{eqn:Dirichlet-to-Neumann} (written as $u'$ in \eqref{eqn:Euler-Lagrange_phi}).
In this section we derive some pointwise estimates that will help to achieve this.

\begin{lemma} \label{lem:Lambda_estimate}
For any $f \in H^{2}(\R)$, any $a \in \R$ and any $R \ge 1$,
\[
|\Lambda f(a + R)| \le \frac{21}{R} \|f\|_{L^2(\R)} + 9\|f'\|_{L^2(a, \infty)} + \|f''\|_{L^2(a, \infty)}.
\]
\end{lemma}

\begin{proof}
We may assume that $a = 0$. Let 
\[
\chi(x_1) = \begin{cases}
0 & \text{if $x_1 \le 0$}, \\
8x_1^2/R^2 & \text{if $0 < x_1 \le R/4$}, \\
1 - 2(1 - 2x_1/R)^2 & \text{if $R/4 < x_1 \le R/2$}, \\
1 & \text{if $x_1 > R/2$}.
\end{cases}
\]
Then $\chi \in C^{1,1}(\R)$ with $|\chi'| \le 4/R$ and $|\chi''| \le 16/R^2$.
We split $\Lambda$ into two operators as in
Remark~\ref{rem:lambda_h1loc}: for $f \in H^{2}(\R)$, let
\[
\Lambda_+ f = \Lambda(\chi f) \quad \text{and} \quad \Lambda_- f = \Lambda ((1 - \chi) f).
\]
Then it follows from Plancherel's theorem that
\[
\|\Lambda_+ f\|_{L^2(\R)} = \|(\chi f)'\|_{L^2(\R)} \le \frac{4}{R} \|f\|_{L^2(0, \infty)} + \|f'\|_{L^2(0, \infty)}.
\]
Moreover,
\[
\|(\Lambda_+ f)'\|_{L^2(\R)}  = \|(\chi f)''\|_{L^2(\R)} \le \frac{16}{R^2} \|f\|_{L^2(0, \infty)} + \frac{8}{R} \|f'\|_{L^2(0, \infty)} + \|f''\|_{L^2(0, \infty)}.
\]
Both inequalities combined imply that
\[
\begin{split}
|\Lambda_+f (R)| & \le \|\Lambda_+ f\|_{L^1(R, R + 1)} + \|(\Lambda_+ f)'\|_{L^1(R, R + 1)} \\
& \le \frac{20}{R} \|f\|_{L^2(0, \infty)} + 9\|f'\|_{L^2(0, \infty)} + \|f''\|_{L^2(0, \infty)}.
\end{split}
\]
For $\Lambda_-f$, we have
\[
\Lambda_- f(R) = \frac{1}{\pi} \int_{-\infty}^{R/2} \frac{(1 - \chi(t)) f(t)}{(t - R)^2} \, dt
\]
by \eqref{eqn:Dirichlet-to-Neumann}. Hence
\[
|\Lambda_- f(R)| \le \frac{1}{\pi} \left(\int_{R/2}^\infty \frac{dt}{t^4}\right)^{1/2} \|f\|_{L^2(\R)} \le R^{-3/2}\|f\|_{L^2(\R)}.
\]
Combining these estimates, we finally obtain
the desired inequality.
\end{proof}

\begin{lemma} \label{lem:estimate_of_u'}
There exists a universal constant $C$ with the following property. Suppose that $\phi \in C^\infty(\R)$
is a solution of \eqref{eqn:Euler-Lagrange_phi} and there exists a number $a \in \R$
such that $\sin \phi \not= 0$ in $(a, \infty)$. Then for $x_1 > a + 1$,
\[
|\Lambda(\cos \phi - k)(x_1)| \le \frac{C}{x_1 - a}\sqrt{\frac{E_h(\cos \phi, \sin \phi)}{\min\{1,h - 1\}}} \quad \text{if $h > 1$}
\]
and
\[
|\Lambda(\cos \phi - k)(x_1)| \le \frac{C}{x_1 - a}\sqrt{E_h(\cos \phi, \sin \phi)} \quad \text{if $h < 1$}.
\]
\end{lemma}

\begin{proof}
Set $m = (\cos \phi, \sin \phi)$ and $f = \cos \phi - k$. Then by Lemma \ref{lem:higher_derivatives}, we have a
universal constant $C_1$ such that for every $R > 0$:
\[
\|f''\|_{L^2(a + R, \infty)} \le \|\phi'\|_{L^4(a + R, \infty)}^2 + \|\phi''\|_{L^2(a + R, \infty)} \le \frac{C_1}{R} \sqrt{E_h(m)}
\]
and
\[
\|f'\|_{L^2(a + R, \infty)} = \|\phi' \sin \phi\|_{L^2(a + R, \infty)} \le \frac{C_1}{R} \sqrt{E_h(m)}.
\]
If $h > 1$, then $W(m)\geq (h-1) |f|$, so that
\[
\|f\|_{L^2(\R)} \le \sqrt{2\|f\|_{L^1(\R)}} \le  \sqrt{2\frac{E_h(m)}{h - 1}}.
\]
If $h < 1$, then
\[
\|f\|_{L^2(\R)} \le \sqrt{2E_h(m)}.
\]
Hence the claim follows from Lemma \ref{lem:Lambda_estimate}.
\end{proof}

The following is another useful estimate based on the cut-off argument in Remark \ref{rem:lambda_h1loc} and 
the proof of Lemma \ref{lem:Lambda_estimate}.

\begin{proposition} \label{prop:Lambda_interpolation}
Let $p \in (1, 2)$ and $q \in [1, \infty)$. Then there exists a constant $C = C(p,q)>0$ such that the following holds true.
Suppose that $f \in \dot{H}^{1/2}(\R) \cap H^{2}_\loc(\R) \cap L^q(\R)$
and $a \in \R$. Then for any $R > 0$,
\[
|\Lambda f(a)| \le \frac{C\left(1 + |\log R|\right)}{R^{1 + 1/p}} \left(R^2 \|f''\|_{L^p(a - R, a + R)} + \|f\|_{L^p(a - R, a + R)}\right) + \frac{C}{R^{1 + 1/q}} \|f\|_{L^q(\R)}.
\]
\end{proposition}

For the proof, we need the following inequalities.

\begin{lemma} \label{lem:log_L^p}
For every $p \in (1, \infty)$ and every $R > 0$,
\[
\int_0^R |\log t|^p \, dt \le pR |\log R|^p + p^p R.
\]
\end{lemma}

\begin{proof}
An integration by parts and H\"older's and Young's inequalities imply
\[
\begin{split}
\int_0^R |\log t|^p \, dt & = R |\log R|^p - p\int_0^R |\log t|^{p - 2} \log t \, dt \\
& \le R |\log R|^p + p R^{1/p} \left(\int_0^R |\log t|^p \, dt\right)^{\frac{p - 1}{p}} \\
& \le R |\log R|^p + p^{p - 1} R + \frac{p - 1}{p} \int_0^R |\log t|^p \, dt,
\end{split}
\]
and the claim follows. 
\end{proof}

\begin{lemma} \label{lem:chi'f'_L^p}
Let $I \subset \R$ be a bounded, open interval. Suppose that
$p \in (1, 2)$ and $f \in W^{2, p}(I)$. Then for
any $\chi \in C_0^{1, 1}(I) \setminus \{0\}$,
\[
\|\chi'f'\|_{L^p(I)} \le \frac{\|\chi'\|_{L^\infty(I)}^2 \|f''\|_{L^p(I)}}{2\|\chi''\|_{L^\infty(I)}} + \frac{2}{p - 1} \|\chi''\|_{L^\infty(I)} \|f\|_{L^p(I)}.
\]
\end{lemma}

\begin{proof}
For $\epsilon > 0$, set $f_\epsilon = \sqrt{f^2 + \epsilon^2}$ and note
that $f_\epsilon' = ff'/f_\epsilon$ and $f_\epsilon'' = ff''/f_\epsilon + \epsilon^2(f')^2/f_\epsilon^3 \ge ff''/f_\epsilon$. Hence
using H\"older's inequality, an integration by parts, and
H\"older's inequality again, we find that
\[
\begin{split}
\int_I |\chi'|^p |f_\epsilon'|^p \, dx_1
& \le \left(\int_I (\chi')^2 (f_\epsilon')^2 f_\epsilon^{p - 2} \, dx_1\right)^{p/2} \left(\int_I f_\epsilon^p \, dx_1\right)^{1 - p/2} \\
& = \left(\frac{1}{1 - p} \int_I \left((\chi')^2 f_\epsilon'' f_\epsilon^{p - 1} + 2 \chi' \chi'' f_\epsilon' f_\epsilon^{p - 1}\right) \, dx_1\right)^{p/2} \left(\int_I f_\epsilon^p \, dx_1\right)^{1 - p/2} \\
& \le \left(\frac{1}{1 - p} \int_I \left((\chi')^2 f'' f f_\epsilon^{p - 2} + 2 \chi' \chi'' f_\epsilon' f_\epsilon^{p - 1}\right) \, dx_1\right)^{p/2} \left(\int_I f_\epsilon^p \, dx_1\right)^{1 - p/2} \\
& \le \left(\frac{1}{p - 1} \|\chi'\|_{L^\infty(I)}^2 \|f''\|_{L^p(I)} + \frac{2}{p - 1} \|\chi''\|_{L^\infty(I)} \|\chi' f_\epsilon'\|_{L^p(I)}\right)^{p/2} \|f_\epsilon\|_{L^p(I)}^{p/2}.
\end{split}
\]
Now Young's inequality yields
\[
\begin{split}
\|\chi' f_\epsilon'\|_{L^p(I)} & \le \left(\|\chi' f_\epsilon'\|_{L^p(I)} + \frac{\|\chi'\|_{L^\infty(I)}^2 \|f''\|_{L^p(I)}}{2\|\chi''\|_{L^\infty(I)}}\right)^{1/2} \left(\frac{2}{p - 1} \|\chi''\|_{L^\infty(I)} \|f_\epsilon\|_{L^p(I)}\right)^{1/2} \\
& \le \frac{1}{2} \|\chi' f_\epsilon'\|_{L^p(I)} + \frac{\|\chi'\|_{L^\infty(I)}^2 \|f''\|_{L^p(I)}}{4\|\chi''\|_{L^\infty(I)}} + \frac{1}{p - 1} \|\chi''\|_{L^\infty(I)} \|f_\epsilon\|_{L^p(I)}.
\end{split}
\]
We conclude that
\[
\|\chi'f_\epsilon'\|_{L^p(I)} \le \frac{\|\chi'\|_{L^\infty(I)}^2 \|f''\|_{L^p(I)}}{2\|\chi''\|_{L^\infty(I)}} + \frac{2}{p - 1} \|\chi''\|_{L^\infty(I)} \|f_\epsilon\|_{L^p(I)}.
\]
The claim now follows from Lebesgue's dominated convergence
theorem.
\end{proof}

\begin{proof}[Proof of Proposition \ref{prop:Lambda_interpolation}]
We may assume without loss of generality that $a = 0$. 
Let $v \in \dot{H}^1(\R_+^2)$ be the harmonic extension of
$f$ to the half-plane, i.e., $v=V(f)$ as defined in \eqref{def:Vm}. By the Poisson formula, we have
\[
v(x_1, x_2) = \frac{x_2}{\pi} \int_{-\infty}^\infty \frac{f(t)}{(t - x_1)^2 + x_2^2} \, dt.
\]
As in the proof of Lemma \ref{lem:Lambda_estimate}, we 
choose a cut-off function $\chi \in C_0^{1, 1}(-R, R)$ with
$0 \le \chi \le 1$ and with $\chi \equiv 1$ in $(-R/2, R/2)$,
such that 
\be
\label{ineg_chi}
|\chi'| \le 4/R \quad \text{and} \quad |\chi''| \le 16/R^2.
\ee
We decompose, as in Remark \ref{rem:lambda_h1loc},
$$v = v_0 + v_1,\quad v_0=V(\chi f), \quad v_1=V\big((1-\chi)f\big);$$ 
that is,
\[
v_0(x_1, x_2) =\frac{x_2}{\pi} \int_{-\infty}^\infty \frac{f(t) \chi(t)}{(t - x_1)^2 + x_2^2} \, dt
\]
and
\[
v_1(x_1, x_2) = \frac{x_2}{\pi} \int_{-\infty}^\infty \frac{f(t) (1 - \chi(t))}{(t - x_1)^2 + x_2^2} \, dt.
\]
By \eqref{egalite}, we have
$$|\Lambda f(0)|= \left|\dd{v}{x_2}(0, 0)\right|\leq \left|\dd{v_0}{x_2}(0, 0)\right|+\left|\dd{v_1}{x_2}(0, 0)\right|.$$

\paragraph{Step 1: estimate for $\dd{v_1}{x_2}(0, 0)$.}
For any $q>1$, we have the estimate 
\[
\begin{split}
\left|\dd{v_1}{x_2}(0, 0)\right| & \le \frac{1}{\pi} \int_{\R \setminus (-R/2, R/2)} \frac{|f(t)|}{t^2} \, dt \\
& \le \frac{1}{\pi} \left(2\int_{R/2}^\infty \frac{dt}{t^{2q/(q - 1)}}\right)^{\frac{q - 1}{q}} \|f\|_{L^q(\R)} =\frac{1}{\pi} \left(\frac{2q - 2}{q + 1}\right)^{\frac{q - 1}{q}} \left(\frac{R}{2}\right)^{- \frac{q + 1}{q}} \|f\|_{L^q(\R)}.
\end{split}
\]
A similar inequality also holds if $q = 1$.

\paragraph{Step 2: estimate for $\dd{v_0}{x_2}(0, 0)$.}
We write $g = \chi f\in H^2(\R)$ with $\supp g\subset [-R,R]$. For $v_0$, we then perform the change of variables $t=x_2 s + x_1$ and
obtain
\[
v_0(x_1, x_2) = \frac{1}{\pi} \int_{-\infty}^\infty \frac{g(x_2 s + x_1)}{s^2 + 1} \, ds.
\]
Hence
\[
\dd{v_0}{x_2}(x_1, x_2) = \frac{1}{\pi} \int_{-\infty}^\infty \frac{s g'(x_2 s + x_1)}{s^2 + 1} \, ds.
\]
As
\[
\frac{d}{ds} \left(\frac{1}{2} \log \left(x_2^2s^2 + x_2^2\right)\right) = \frac{s}{s^2 + 1},
\]
an integration by parts yields
\[
\begin{split}
\dd{v_0}{x_2}(x_1, x_2) & = -\frac{x_2}{2\pi} \int_{-\infty}^\infty g''(x_2 s + x_1) \log(x_2^2 s^2 + x_2^2) \, ds \\
& = - \frac{1}{2\pi} \int_{-\infty}^\infty g''(t) \log((t - x_1)^2 + x_2^2) \, dt.
\end{split}
\]
In particular,
\[
\dd{v_0}{x_2}(0, 0) = - \frac{1}{\pi} \int_{-\infty}^\infty g''(t) \log |t| \, dt,
\]
which implies, for $p\in (1,2)$, that
\[
\left|\dd{v_0}{x_2}(0, 0)\right| \le \frac 1\pi \left(2\int_0^R |\log t|^{p/(p - 1)} \, dt\right)^{\frac{p - 1}{p}} \|g''\|_{L^p(\R)}.
\]
As a consequence of this and Lemma \ref{lem:log_L^p}, we obtain a constant $C_1 = C_1(p)$ such that
\[
\left|\dd{v_0}{x_2}(0, 0)\right| \le C_1\left(1 + |\log R|\right) R^{(p - 1)/p} \|g''\|_{L^p(\R)}.
\]
It remains to estimate the $L^p$-norm of $g''$.
To this end, we observe that $g'' = \chi f'' + 2 \chi' f' + \chi'' f$.
Hence
\[
\|g''\|_{L^p(\R)} \le \|f''\|_{L^p(-R, R)} + 2\|\chi' f'\|_{L^p(\R)} + \|\chi''\|_{L^\infty(\R)} \|f\|_{L^p(-R, R)}.
\]
Lemma \ref{lem:chi'f'_L^p} provides an estimate for the
second term. Using \eqref{ineg_chi}, we then see that
there exists a constant $C_2 = C_2(p)$ satisfying
\[
\|g''\|_{L^p(\R)} \le 2\|f''\|_{L^p(-R, R)} + \frac{C_2}{R^2} \|f\|_{L^p(-R, R)}.
\]
Now it suffices to combine the above inequalities.
\end{proof}

\section{Analysis of the Euler-Lagrange equation} \label{sect:decay}

We now analyse the Euler-Lagrange
equation for minimisers $m=(\cos \phi, \sin \phi)$ of $E_h$ in $\A_h(d)$ for a given $d \in \N$ in the case $h > 1$
and for $d = \alpha/\pi$ or $d = 1 - \alpha/\pi$ in the case $h < 1$.
Of particular interest is the rate of decay of $m_1$ near $\pm \infty$.

\subsection{Exponential decay for $h > 1$}

We proceed to establish exponential decay of minimisers $\phi$ and its
derivatives. To this end, we
first prove the following lemmas.

\begin{lemma} \label{lem:decreasing}
Let $h>1$ and $a > 0$, and let $\phi \colon \R \to \R$ be a smooth function such that $1-\cos \phi \in \dot{H}^{1/2}(\R)$ and
$$\begin{cases}
&\text{$\phi$ is solution of \eqref{eqn:Euler-Lagrange_phi} in $(a, \infty)$,}\\
& 0<\phi<\pi \quad \text{and} \quad |\Lambda(1 - \cos \phi)| \le \frac{h - 1}{2} \quad \text{in } (a, \infty).
\end{cases}$$
Then $\phi' \le 0$ in $[a, \infty)$.
\end{lemma}

\begin{proof}
Suppose, by way of contradiction, that there exists $b \ge a$
with $\phi'(b) > 0$. Then there exists $c > b$ such that
$\phi' > 0$ in $[b, c)$ and $\phi'(c) = \phi'(b)/2$. As $\sin \phi>0$ and $|\Lambda(1 - \cos \phi)| \le \frac{h - 1}{2}$ in $(a, \infty)$, equation \eqref{eqn:Euler-Lagrange_phi}
implies
\begin{equation} \label{eqn:inequality_second_derivative}
\phi'' \ge \frac{1}{2} (h - \cos \phi) \sin \phi \quad \text{in $(a, \infty)$}.
\end{equation}
Hence
\[
\frac{d}{dx_1} (\phi'(x_1))^2 \ge  [(h - \cos \phi) \sin \phi] \, \phi'> 0 \quad \text{in $(b, c)$}.
\]
It follows that $\phi'(c) > \phi'(b) > 0$,
in contradiction to the choice of $c$.
\end{proof}

\begin{proposition} \label{prop:exponential_decay}
Let $h > 1$. Then there exists a constant $c > 0$
with the following property. Let $a > 0$ and let $\phi\colon \R\to \R$
be a smooth function such that 
$1-\cos \phi \in \dot{H}^{1/2}(\R)$ and
$$\begin{cases}
&\text{$\phi$ is solution of \eqref{eqn:Euler-Lagrange_phi} in $(a, \infty)$,}\\
& 0<\phi\leq 1 \quad \text{and} \quad |\Lambda(1 - \cos \phi)| \le \frac{h - 1}{2} \quad \text{in } (a, \infty),\\
& \lim_{x_1 \to \infty} \phi(x_1) = 0.
\end{cases}$$
Then
\[
\phi(x_1) \le e^{c(a - x_1)} \quad \text{for all $x_1 \ge a$.}
\]
\end{proposition}

\begin{remark}
It will not be necessary to know the value of $c$ explicitly,
but we will prove the inequality for $c = \gamma \sqrt{h - 1}$,
where $\gamma$ is the constant introduced in Lemma \ref{lem:gamma}.
\end{remark}

\begin{proof}
Under the hypotheses of the lemma, equation
\eqref{eqn:Euler-Lagrange_phi} gives rise to the inequality
\eqref{eqn:inequality_second_derivative} in $(a, \infty)$
again. As $\phi' \le 0$ in $[a, \infty)$ by Lemma \ref{lem:decreasing},
this implies that $\limsup_{x_1 \to \infty} \phi'(x_1) \leq 0$ and
\[
\frac{d}{dx_1} \left((\phi'(x_1))^2 - \frac{1}{2} (h - \cos \phi(x_1))^2\right) \le 0 \quad \text{in } (a, \infty).
\]
As $\lim_{x_1 \to \infty} \phi(x_1) = 0$, we deduce $\limsup_{x_1 \to \infty} \phi'(x_1) = 0$ and
$\lim_{x_1 \to \infty} \cos \phi = 1$, so it follows that
\[
(\phi'(x_1))^2 \ge \frac{1}{2} \left(\cos^2 \phi(x_1) - 2h\cos \phi(x_1) + 2h - 1\right) = W(\cos \phi(x_1), \sin \phi(x_1)) \quad \text{for all $x_1 \ge a$}.
\]
Therefore,
\[
\phi'(x_1) \le - \sqrt{W(\cos \phi(x_1), \sin \phi(x_1))} \quad \text{for all $x_1 \ge a$}.
\]
Since $W(\cos \phi, \sin \phi) \ge c^2 \phi^2$ for $c = \gamma \sqrt{h - 1}$
by Lemma \ref{lem:gamma}, we conclude that $\phi' \le -c \phi$
in $[a, \infty)$, from which we finally obtain the desired inequality.
\end{proof}

For minimisers in $\A_h(d)$ with $d \in \Z$, we can now prove exponential decay at $\pm \infty$.
For convenience, we consider \emph{negative} winding numbers in the statement of the
next result, but of course we immediately obtain a statement for positive
winding numbers as well.

\begin{theorem}[Exponential decay for $h > 1$] \label{thm:exponential_decay}
Let $h > 1$, $d \in \N$, $\beta < 2$, and let
$m = (\cos \phi, \sin \phi) \in \A_h(-d)$ be a minimiser of $E_h$ in $\A_h(-d)$ such that
$$\lim_{x_1 \to \infty} \phi(x_1) = 0.$$
Then there exist $a \in \R$ and $c, C > 0$ such that for all $x_1 \ge a$: {$\phi'(x_1)\leq 0$ and}
\begin{equation} \label{eqn:exponential_decay}
\max\{|\phi(x_1)|, |\phi'(x_1)|, |\phi''(x_1)|\} \le e^{c(a - x_1)}
\end{equation}
and
\[
|\Lambda (m_1 - 1) (x_1)| \le \frac{C}{(x_1 - a)^\beta}.
\]
\end{theorem}

\begin{proof} By Proposition \ref{prop:regularity}, we know that $\phi$ is smooth.
By the hypothesis and Lemma \ref{lem:passages}, there exists
$a' \geq 1$ such that 
$$0 < \sin \phi \leq \phi \le 1 \quad \text{in } [a', \infty).$$ 
(The fact that the degree of $m$ is $-d<0$ is essential for the
positive sign of $m_2=\sin \phi$ near $+\infty$.) Moreover, by Lemma \ref{lem:estimate_of_u'},
we may assume that 
$$|\Lambda (1 - \cos \phi)| \le \frac{h - 1}{2} \quad \text{in } [a', \infty)$$ 
as well.
Hence, {Lemma \ref{lem:decreasing} implies that $\phi$ is monotone} in $[a', \infty)$; also, we may apply Proposition \ref{prop:exponential_decay} and we obtain
a constant $c > 0$ such that
\[
\phi(x_1) \le e^{c(a' - x_1)} \quad \text{for } x_1\geq a'.
\]
Using equation \eqref{eqn:Euler-Lagrange_phi}, we then obtain
\[
|\phi''(x_1)| \le \frac{3h - 1}{2} \phi(x_1) \le \frac{3h - 1}{2} e^{c(a' - x_1)}  \quad \text{for } x_1\geq a'.
\]
If $a'' \ge a'$ is chosen sufficiently large, then it follows that
$
|\phi''(x_1)| \le e^{c(a'' - x_1)}
$
for $x_1 \ge a''$. Since $\liminf_{x_1 \to \infty} |\phi'(x_1)| = 0$ (because $\phi(x_1)\to 0$ as $x_1 \to \infty$),
this implies
\[
|\phi'(x_1)| \le \int_{x_1}^\infty |\phi''(t)| \, dt \le \frac{1}{c} e^{c(a'' - x_1)} \quad \text{for } x_1\geq a''.
\]
Choosing $a$ sufficiently large, we obtain inequality \eqref{eqn:exponential_decay}.

It remains to establish the decay of  $\Lambda (m_1 - 1)$ at $\infty$. 
Lemma \ref{lem:estimate_of_u'} already gives the decay $1/{x_1}$ as $x_1\to \infty$. In order to improve it, 
we may assume without loss of
generality that inequalities similar to \eqref{eqn:exponential_decay} hold for
$2\pi d - \phi(x_1)$ and for the derivatives $\phi'(x_1)$
and $\phi''(x_1)$ when $x_1 \le -a''$
(because the behaviour
of $\phi$ as $x_1 \to -\infty$ is similar, albeit with
limit $2\pi d$). 
Fix $p \in (1, 2)$ such that $\beta<1+1/p$. Then it follows immediately that
\[
\|\cos \phi - 1\|_{L^p(\R)} \le C_1
\]
for a constant $C_1$ that depend only on $p$, $c$ and $a''$.
Moreover, for every $x_1 \ge 2a''$ and $R=\frac{x_1-a''}{2}$, we have
the inequality
\[
\int_{x_1 - R}^{x_1 + R} \left(|\phi''(t)|^p + |\phi'(t)|^{2p}\right) \, dt \le C_2 e^{cp(a'' - x_1)/2},
\]
where $C_2 = C_2(p, c, a'')$.
We apply Proposition \ref{prop:Lambda_interpolation} for $f=1-\cos \phi$ and $q=p$. Since $|f''|\leq |\phi''|+|\phi'|^2$, then
there exists a constant $C_3$ with
\[
|\Lambda (1-\cos \phi) (x_1)| \le \frac{C_3(1 + |\log (x_1-a'')|)}{(x_1 - a'')^{1 + 1/p}}[1 + (x_1 - a'')^2 e^{c(a'' - x_1)/2}]
\]
for all $x_1 \ge 2a''$. If we choose $a \ge 2a''$ large enough, then
the desired inequality follows for all $x_1\geq a$.
\end{proof}

\subsection{The linearised equation for $h < 1$}

When $h=\cos \alpha \in [0, 1)$ with $\alpha \in (0, \frac \pi 2]$, we will not obtain exponential decay of the
minimising profile, because the contribution of the non-local
differential operator in \eqref{eqn:Euler-Lagrange_phi}
is no longer dominated by the local terms. Our analysis here
is motivated by the analysis of Chermisi-Muratov \cite{Chermisi-Muratov:13}
for the winding numbers $\frac{\alpha}{\pi}$ and $1 - \frac{\alpha}{\pi}$.
An important tool is the fundamental solution of the linearisation
of \eqref{eqn:Euler-Lagrange_phi} about the trivial solution $\phi_0=\alpha$, which is
calculated in the aforementioned work. The paper also gives estimates
for the fundamental solution, which we improve somewhat here.

We consider the differential operator $L$, given by\footnote{The linearisation
of \eqref{eqn:Euler-Lagrange_phi} about $\phi_0=\alpha$ is then given by $L\left(x_1\mapsto \sin^2\alpha \psi\left(\frac{x_1}{\sin \alpha}\right)\right)$.}
\begin{equation} \label{eqn:linearised}
L\psi = -\psi'' + \psi - \sin \alpha \, \Lambda \psi.
\end{equation}
The fundamental solution $G_\alpha$ for the equation $L\psi = 0$ (satisfying $L G_\alpha=\delta_0$, where $\delta_0$ is the Dirac measure at $0$) is computed, using the Fourier transform and contour integration,
by Chermisi--Muratov \cite[Lemma A.1]{Chermisi-Muratov:13}. It is
\begin{equation} \label{eqn:fundamental_solution}
G_\alpha(x_1) = \frac{1}{2\pi}\int_{\R} \frac{e^{i\xi x_1}\, d\xi }{\xi^2+1+ |\xi |\sin \alpha}=\frac{\sin \alpha}{\pi} \int_0^\infty \frac{t e^{-t|x_1|}}{t^2 \sin^2 \alpha + (t^2 - 1)^2} \, dt \quad \text{for all } x_1\in \R.
\end{equation}
That is, for a solution $g \in H^{2}(\R)$ of the equation $Lg = f$
with $f \in L^2(\R)$, we have
\[
g = G_\alpha * f.
\]

\begin{lemma} \label{lem:fundamental_solution}
There exists a constant $C>0$ such that for any
$\alpha \in (0, \frac{\pi}{2}]$, the fundamental
solution $G_\alpha$ of the operator $L$ defined in
\eqref{eqn:linearised} satisfies, for all $x_1 \not= 0$, the inequalities
\[
0 \le G_\alpha(x_1) \le \frac{C \sin \alpha}{1 + x_1^2} + Ce^{-|x_1|/2}
\]
and
\[
0 \le - \frac{x_1}{|x_1|} G_\alpha'(x_1) \le \frac{C \sin \alpha}{1 + |x_1|^3} + Ce^{-|x_1|/2}
\]
and\footnote{In the sense of distributions, we have $G_\alpha''\in \delta_0+L^2(\R)$, so we estimate the diffuse part of $G_\alpha''$ here 
(still denoted $G_\alpha''$).}
\[
0 \le G_\alpha''(x_1) \le C\sin \alpha \frac{\big|\log |x_1|\, \big|}{1 + x_1^4 \big|\log |x_1|\, \big|}+ Ce^{-|x_1|/2}.
\]
\end{lemma}

\begin{proof} By definition, the Fourier transform of $G_\alpha$ is given by
$${\cal F} G_\alpha(\xi)=\frac{1}{\sqrt{2\pi}} \frac{1}{\xi^2+1+|\xi|\sin \alpha}, \quad \xi\in \R,$$
which immediately implies that $G_\alpha\in H^1(\R)$ with 
$$\|G_\alpha\|_{{H}^1(\R)}\leq C_1$$
for a constant $C_1>0$ independent of $\alpha$. 
As $LG_\alpha=\delta_0$, we deduce that $G''_\alpha\in\delta_0+L^2(\R)$ (as a distribution). As a function, however, $G_\alpha$ is smooth
at every $x_1 \not= 0$ with
\[
G_\alpha'(x_1) = -\frac{x_1}{|x_1|} \frac{\sin \alpha}{\pi} \int_0^\infty \frac{t^2 e^{-t|x_1|}}{t^2 \sin^2 \alpha + (t^2 - 1)^2} \, dt
\]
and
\[
G_\alpha''(x_1) = \frac{\sin \alpha}{\pi} \int_0^\infty \frac{t^3 e^{-t|x_1|}}{t^2 \sin^2 \alpha + (t^2 - 1)^2} \, dt.
\]

\paragraph{Step 1: estimates for $|x_1| \ge 1$.} We have
\[
\int_0^{1/2} \frac{te^{-t|x_1|}}{t^2 \sin^2 \alpha + (t^2 - 1)^2} \, dt \le 4\int_0^\infty te^{-t|x_1|} \, dt = \frac{4}{x_1^2} \int_0^\infty se^{-s} \, ds.
\]
If $\alpha \le \frac{\pi}{6}$, then
\[
\int_{1/2}^{1 - \sin \alpha} \frac{te^{-t|x_1|}}{t^2 \sin^2 \alpha + (t^2 - 1)^2} \, dt \le e^{-|x_1|/2} \int_{-\infty}^{1 - \sin \alpha} \frac{dt}{(t - 1)^2} = \frac{e^{-|x_1|/2}}{\sin \alpha}
\]
and
\[
\int_{1 + \sin \alpha}^\infty \frac{te^{-t|x_1|}}{t^2 \sin^2 \alpha + (t^2 - 1)^2} \, dt \le e^{-|x_1|} \int_{1 + \sin \alpha}^\infty \frac{dt}{(t - 1)^2} = \frac{e^{-|x_1|}}{\sin \alpha}
\]
as well. Moreover,
\[
\int_{1 - \sin \alpha}^{1 + \sin \alpha} \frac{te^{-t|x_1|}}{t^2 \sin^2 \alpha + (t^2 - 1)^2} \, dt \le \frac{e^{-|x_1|/2}}{\sin^2 \alpha} \int_{1 - \sin \alpha}^{1 + \sin \alpha} \frac{dt}{t} \le \frac{4e^{-|x_1|/2}}{\sin \alpha}.
\]
If $\alpha > \frac{\pi}{6}$, then we observe instead that
\[
\int_{1/2}^\infty \frac{te^{-t|x_1|}}{t^2 \sin^2 \alpha + (t^2 - 1)^2} \, dt \le e^{-|x_1|/2} \int_{1/2}^\infty \frac{t}{t^2/4 + (t^2 - 1)^2} \, dt.
\]
The integral on the right-hand side converges,
and the inequalities for $G_\alpha$ follow immediately.
For $G_\alpha'$ and $G_\alpha''$, we can use the same arguments
when $|x_1| \ge 1$.

\paragraph{Step 2: estimates for $|x_1| < 1$.}
For $G_\alpha$, we know that $\|G_\alpha\|_{H^1(\R)}$
is bounded uniformly in $\alpha$. We conclude that
$|G_\alpha(x_1)|$ is bounded uniformly in $\alpha \in (0, \frac{\pi}{2}]$
and $x_1 \in [-1, 1]$.

For $G_\alpha'$ and $G_\alpha''$, we first observe that
\[
|G_\alpha'(x_1)| \le 2G_\alpha(x_1) + \frac{4\sin \alpha}{\pi} \int_2^\infty \, \frac{dt}{t^2} = 2G_\alpha(x_1) + \frac{2\sin \alpha}{\pi}
\]
and
\[
|G_\alpha''(x_1)| \le 4G_\alpha(x_1) + \frac{4\sin \alpha}{\pi} \int_2^\infty t^{-1} e^{-t|x_1|} \, dt.
\]
Since
\[
\int_2^\infty t^{-1} e^{-t|x_1|} \, dt \le \int_2^{2/|x_1|} \frac{dt}{t} + \int_2^\infty e^{-s} \, \frac{ds}{s} = \log \frac{1}{|x_1|} + \int_2^\infty e^{-s} \, \frac{ds}{s},
\]
the desired inequalities follow for $|x_1| < 1$ as well.
\end{proof}

A considerable part of the subsequent analysis is based
on the decay behaviour of $G_\alpha$ and its derivatives,
together with the following principle: if
$G, \psi \in L^1(\R) \cap L^\infty(\R)$, then
\[
(G * \psi)(x_1) = \int_{x_1/2}^\infty (G(t) \psi(x_1 - t) + G(x_1 - t) \psi (t)) \, dt;
\]
therefore,
\begin{equation} \label{eqn:convolution}
|(G * \psi)(x_1)| \le \|G\|_{L^\infty(x_1/2, \infty)} \|\psi\|_{L^1(\R)} + \|G\|_{L^1(\R)} \|\psi\|_{L^\infty(x_1/2, \infty)}.
\end{equation}

\subsection{Polynomial decay for $h < 1$}

For $h<1$, we will prove polynomial decay for minimisers of $E_h$ in
$\A_h(d)$ for $d\in \{\frac{\alpha}{\pi}, 1 - \frac{\alpha}{\pi}\}$. The following decay estimates improve the results of Chermisi-Muratov \cite[Lemma 5]{Chermisi-Muratov:13}. In particular, we prove cubic and quartic decay of $f'$ and $f''$, respectively, as well as a new $L^1$-estimate for
$f$, which is fundamental for the proofs of our main results stated in Section \ref{sec:main}.

\begin{theorem} \label{thm:decay_h<1}
There exist universal constants $c, C >0$ with the following property.
For every $h = \cos \alpha$ with $\alpha \in (0, \frac{\pi}{2}]$,
there exists a unique increasing, odd function $\phi \colon \R \to \R$
such that $m = (\cos \phi, \sin \phi)$ is
a minimiser of $E_h$ in $\A_h(\alpha/\pi)$. Furthermore, the function
$f = \cos \phi - \cos \alpha$ satisfies
\[
0<f(x_1) \le \frac{C}{x_1^2}, \quad |f'(x_1)| \le \frac{C}{x_1^3}, \quad |f''(x_1)| \le \frac{C}{x_1^4}, \quad \text{and} \quad |\Lambda f(x_1)| \le \frac{C \alpha}{x_1^2} \quad \text{for all $x_1 \ge \frac{c}{\alpha}$}
\]
and also
\[
\|f\|_{L^1(\R)} \le C \alpha.
\]
\end{theorem}

\begin{proof}
We use various universal constants in this proof, and we will
abuse notation and indiscriminately use the symbol $C$ for most
of them. The existence of a unique symmetric minimiser follows by symmetrization via rearrangement as proved in the
works of Melcher \cite{Me1} and Chermisi-Muratov
\cite{Chermisi-Muratov:13} {(see also Lemma \ref{lem:symmetry} above)}. Moreover, $\phi$ is increasing with $\phi(\R)=(-\alpha, \alpha)$. By the symmetry, the function
$\phi$ is odd. Thus it suffices to prove the
inequalities. To this end, we first rescale the solutions.

\paragraph{Step 1: rescaling.} Set $f = \cos \phi - \cos \alpha$ and
\[
g(x_1) = \frac{1}{\sin^2 \alpha} f\left(\frac{x_1}{\sin \alpha}\right).
\]
As $0<f\leq 1-\cos \alpha$ in $\R$, we deduce that $0<g\leq 1$.
Moreover, as $f$ satisfies \eqref{eqn:Euler-Lagrange_m_1} away from $x_1=0$,
we know that $g$ is a solution of the equation
\[
g'' = - \frac{(g')^2(g \sin^2 \alpha + \cos \alpha)}{1 - 2g \cos \alpha - g^2 \sin^2 \alpha} + (g - \sin \alpha \Lambda g)(1 - 2g \cos \alpha - g^2 \sin^2 \alpha), \quad x_1\neq 0.
\]
Define the operator $L$ as in \eqref{eqn:linearised}.
Then we can write the equation in the form
\begin{equation} \label{eqn:Euler-Lagrange_g}
Lg = A (g')^2 + gB (g - \sin \alpha \Lambda g) \quad \text{in } \R\setminus \{0\},
\end{equation}
where
\[
B = 2 \cos \alpha + g \sin^2 \alpha=\cos\alpha+\cos \phi(\frac{\blank}{\sin \alpha})
\]
and
\[
A = \frac{g \sin^2 \alpha + \cos \alpha}{1 - gB}=\sin^2\alpha 
\frac{\cos \phi(\frac{\blank}{\sin \alpha})}{\sin^2 \phi(\frac{\blank}{\sin \alpha})} \quad \text{in } \R\setminus \{0\}.
\]
The function $B$ is bounded (with $|B|\leq 2$ in $\R$), 
whereas $A$ is unbounded for every $\alpha\in (0, \frac \pi 2]$ (since $A(x_1)\to \infty$ as $x_1\to 0$) and $A > 0$ for $x_1\neq 0$. However, for any $x_1$ such that
$|\phi(\frac{x_1}{\sin \alpha})| \ge \frac{\alpha}{2}$, we have $A(x_1) \le C$.

\paragraph{Step 2: prove $L^2$-estimates.}
We want to show that
\begin{equation} \label{eqn:L^2_norms_involving_g}
\left(\int_{-\infty}^\infty A(g')^2 \, dx_1\right)^{1/2} + \|g\|_{L^2(\R)} + \alpha \|g'\|_{L^2(\R)} \le C.
\end{equation}
To this end, we first compute
\[
A(x_1) (g'(x_1))^2 = \frac{\cos \phi(\frac{x_1}{\sin \alpha})}{\sin^4 \alpha} \left(\phi'\left(\frac{x_1}{\sin \alpha}\right)\right)^2, \quad x_1\neq 0.
\]
Therefore,
\[
\int_{-\infty}^\infty A(g')^2 \, dx_1 \le \frac{1}{\sin^3 \alpha} \int_{-\infty}^\infty (\phi')^2 \, dx_1 \le \frac{2E_h(m)}{\sin^3 \alpha},
\]
where $m = (\cos \phi, \sin \phi)$.
Furthermore, we compute
\[
\|g\|_{L^2(\R)}^2 = \frac{\|f\|_{L^2(\R)}^2}{\sin^3 \alpha} \le \frac{2E_h(m)}{\sin^3 \alpha}
\]
and similarly
\[
\|g'\|^2_{L^2(\R)} = \frac{\|f'\|_{L^2(\R)}^2}{\sin^5 \alpha} \le \frac{2 E_h(m)}{\sin^5 \alpha}.
\]
Using Lemma \ref{lem:cubic_growth}, we obtain \eqref{eqn:L^2_norms_involving_g}.

\paragraph{Step 3: prove preliminary pointwise estimates.}
Next we want to establish the following inequalities: 
\begin{alignat}{2}
\label{eqn:initial_decay}
0<g(x_1) &\leq \frac{C}{\sqrt{|x_1|}} & \quad & \text{for any $x_1 \neq 0$}, \\
\label{estim_g'}
|g'(x_1)| &\leq \frac{C}{\sqrt{\alpha} |x_1|}  && \text{for any $x_1 \neq 0$,} \\
\label{eqn:estim_Lambda_g}
|(\Lambda g)(x_1)| & \leq \frac{C}{\sqrt{\alpha} |x_1|} && \text{for $|x_1| > \sin \alpha$}.
\end{alignat}

For the proof of \eqref{eqn:initial_decay}, we will in fact
show that $g(x_1) \le x_1^{-1/2} \|g\|_{L^2(\R)}$ for $x_1 > 0$.
The inequality then follows by the symmetry and \eqref{eqn:L^2_norms_involving_g}. Assume, for contradiction,
that there exists $x_1>0$ with $g(x_1)>x_1^{-1/2} \|g\|_{L^2(\R)}$. Then for every $t\in (0, x_1)$, we have $g(t)\geq g(x_1)>x_1^{-1/2} \|g\|_{L^2(\R)}$, because $g$ is non-increasing. Therefore, 
$$\int_0^{x_1} g^2\, dt> \|g\|_{L^2(\R)}^2 \int_0^{x_1} \frac{1}{x_1}\, dt=\|g\|^2_{L^2(\R)},$$
which is a contradiction.

As $\phi(0)=0$ and $\phi$ is increasing, we have $0<\phi<\alpha$ for $x_1>0$.
Thus we may use Lemma \ref{lem:higher_derivatives} (for $a=0$ and
with $R/\sin \alpha$ instead of $R$) and Lemma \ref{lem:cubic_growth}
to conclude that
\[
\int_{\frac{R}{\sin \alpha}}^\infty \left((\phi'')^2 + (\phi')^2 \sin^2 \phi + (\phi')^4\right) \, dx_1 \le \frac{C\sin^2 \alpha \, E_h(m)}{R^2}\leq  
\frac{C\sin^5 \alpha}{R^2} \quad \text{for any $R > 0$.}
\]
Hence
\begin{equation} \label{eqn:L^2_norm_of_g''}
\int_{R}^\infty \left(\sin^2 \alpha (g'')^2 + (g')^2 \right) \, dx_1 \le \frac{C}{R^2} \quad \text{for any $R > 0$.}
\end{equation}
In particular, the Cauchy-Schwartz inequality implies, for every $t>R$,
that
$$\sin \alpha \left|(g'(R))^2 - (g'(t))^2\right|\leq 2 \sin \alpha \int_R^t |g' g''|\, dx_1\leq \frac{C}{R^2}.$$
As $g'\in L^2(\R)$, we know that $\liminf_{t\to \infty} |g'(t)|=0$;
so \eqref{estim_g'} follows.
We finally apply Lemmas \ref{lem:estimate_of_u'} and \ref{lem:cubic_growth} to obtain
\[
|(\Lambda g)(x_1)|=\frac{1}{\sin^3\alpha} |(\Lambda f)(\frac{x_1}{\sin \alpha})| \le \frac{C}{\sqrt{\alpha} |x_1|} \quad \text{for } |x_1| > \sin \alpha,
\]
which is \eqref{eqn:estim_Lambda_g}.

We also note that as a consequence of \eqref{eqn:initial_decay}, there exists a constant $a \geq 1$ (independent of $\alpha$)
such that 
\be
\label{def_a}
\phi\left(\frac{x_1}{\sin \alpha}\right) \ge \frac{\alpha}{2} \text{ (and hence $A(x_1) \le C$) whenever $x_1 \ge a$}.
\ee

\paragraph{Step 4: improve the decay.} We now show that
\eqref{eqn:initial_decay} can be improved as follows: 
\[
g(x_1) \le \frac{C}{|x_1|} \quad \text{for $|x_1| \ge {2a}$.}
\]
For this purpose, we use the fact that $g = G_\alpha * Lg$. As $|B|\leq 2$ and $g>0$ in $\R$, we have
\begin{equation} \label{eqn:g_estimate}
g \le G_\alpha * A(g')^2 + 2G_\alpha * g^2 + 2\alpha G_\alpha * g|\Lambda g| \quad \text{for } x_1\neq 0.
\end{equation}
Applying an inequality of the type of \eqref{eqn:convolution}, we find
\begin{align*}
(G_\alpha *A(g')^2 )(x_1) &= \int_{x_1/2}^\infty \bigg(G_\alpha(t) [A(g')^2](x_1 - t) + G_\alpha(x_1 - t)  [A(g')^2](t)\bigg) \, dt\\
&\leq  \|G_\alpha\|_{L^\infty(x_1/2, \infty)} \|A(g')^2\|_{L^1(\R)} + \|G_\alpha\|_{L^\infty(\R)} \|A(g')^2\|_{L^1(x_1/2, \infty)}.
\end{align*}
By \eqref{def_a}, we have $|A(x_1)| \le C$ for $|x_1| \ge a$.
Hence when $|x_1| \ge 2a$, Lemma \ref{lem:fundamental_solution}, together with
\eqref{eqn:L^2_norms_involving_g} and
\eqref{eqn:L^2_norm_of_g''}, implies that
\[
(G_\alpha * A(g')^2)(x_1) \le \frac{C}{x_1^2}.
\]
Similarly, we use \eqref{eqn:convolution} to estimate the
other two terms in \eqref{eqn:g_estimate}. Owing to
\eqref{eqn:L^2_norms_involving_g} and \eqref{eqn:initial_decay},
we obtain
\[
|(G_\alpha * g^2)(x_1)| \le \frac{C}{|x_1|}, \quad x_1\neq 0.
\]
Because
\[
\|g \Lambda g\|_{L^1(\R)} \le \|g\|_{L^2(\R)} \|g'\|_{L^2(\R)} \le \frac{C}{\alpha}
\]
by \eqref{eqn:L^2_norms_involving_g} and
\[
g(x_1) |(\Lambda g)(x_1)| \le \frac{C}{\sqrt{\alpha} |x_1|^{3/2}}, \quad x_1\neq 0,
\]
by \eqref{eqn:initial_decay} and \eqref{eqn:estim_Lambda_g},
we also have
\[
|\alpha (G_\alpha * g|\Lambda g|)(x_1)| \le \frac{C}{|x_1|^{3/2}}, \quad x_1\neq 0.
\]
Therefore, the desired decay for $g$ follows when $|x_1| \ge {2a}$.

\paragraph{Step 5: conclusion.} 
We can use the conclusion of Step 4 to improve the above estimates again.
Namely, we find that
\[
|(G_\alpha * g^2)(x_1)| \le \frac{C}{x_1^2}
\]
and
\[
|\alpha (G_\alpha * g\Lambda g)(x_1)| \le \frac{C}{x_1^2}
\]
for $|x_1| \ge 4a$. Hence
\be
\label{quadrat_decay}
0<g(x_1) \le \frac{C}{x_1^2} \quad \text{for $|x_1| \ge 4a$.}
\ee
Using the formula
\begin{equation} \label{eqn:convolution2}
g' = G_\alpha' * Lg
\end{equation}
and taking advantage of \eqref{quadrat_decay},
we repeat the arguments from Step 4 to obtain, for $|x_1| \ge 8a$,
$$|(G'_\alpha * A(g')^2)(x_1)| \le \frac{C}{x_1^2}, \quad |(G_\alpha * g^2)(x_1)| \le \frac{C}{|x_1|^3},
\quad |\alpha (G_\alpha * g|\Lambda g|)(x_1)| \le \frac{C}{|x_1|^{3}}.$$
Therefore,
\[
|g'(x_1)| \le \frac{C}{|x_1|^2} \quad \text{for $|x_1| \ge 8a$}.
\]
Using this estimate, we obtain
$$\int_{|x_1|/2}^\infty A (g')^2\, dt\leq \frac{C}{|x_1|^3}, \, \quad  |x_1|\geq 16 a,$$
so that $|(G'_\alpha * A(g')^2)(x_1)| \le \frac{C}{x_1^3}$, and finally,
\[
|g'(x_1)| \le \frac{C}{|x_1|^3} \quad \text{for $|x_1| \ge 16a$}.
\]
As $g'' = G_\alpha'' * Lg$, the same method\footnote{Note that $G''_\alpha$ does not belong to $L^\infty$ (by Lemma \ref{lem:fundamental_solution}) so that we can only use $L^1$ estimates near $x_1=0$.} implies, for $|x_1| \ge 32a$, that
\begin{align*}
|(G''_\alpha * A(g')^2)(x_1)| &\le \|G''_\alpha\|_{L^\infty(x_1/2, \infty)} \|A(g')^2\|_{L^1(\R)} + \|G''_\alpha\|_{L^1(\R)} \|A(g')^2\|_{L^\infty(x_1/2, \infty)}
\leq \frac{C}{x_1^4},\\ 
|(G''_\alpha * g^2)(x_1)| &\le \frac{C}{|x_1|^4},
\quad |\alpha (G''_\alpha * g|\Lambda g|)(x_1)| \le \frac{C}{|x_1|^{3}}.
\end{align*}
This in turn yields
$$|g''(x_1)| \le \frac{C}{|x_1|^3}, \quad |x_1|\geq 32 a.$$ In order to obtain the desired quartic power decay of $g''$, we need to improve the estimate of $g|\Lambda g|$. To this end, we use Proposition
\ref{prop:Lambda_interpolation}
(applied with $p$ sufficiently close to $1$, $q=1$ and $R=x_1/2$).
We find that
\[
|\Lambda g(x_1)| \le \frac{C \big|\log |x_1|\big|}{|x_1|^{2}}
\]
for $|x_1| \ge 64 a$, using the fact that
$\|g\|_{L^1(\R)} \le C$ (because $g$ is bounded and satisfies \eqref{quadrat_decay}). Hence
\[
g(x_1) |\Lambda g(x_1)| \le \frac{C\big|\log |x_1|\big|}{|x_1|^{4}}, \quad \text{as well as} \quad  |\alpha (G''_\alpha * g|\Lambda g|)(x_1)| \le  \frac{C\big|\log |x_1|\big|}{|x_1|^{4}},
\]
which yields
\[
|g''(x_1)| \le  \frac{C\big|\log |x_1|\big|}{|x_1|^{4}}
\]
for $|x_1| \ge 128 a$. Applying Proposition
\ref{prop:Lambda_interpolation} again, we obtain 
$$|\Lambda g(x_1)| \le \frac{C}{x_1^2},$$
leading to
\[
|g''(x_1)| \le \frac{C}{x_1^4}  \quad \text{for $|x_1| \ge 256a$}.
\]
Now the inequalities for $f$ follow by rescaling.
\end{proof}

We also state a similar statement for minimisers in the set $\A_h(1 - \alpha/\pi)$,
but we are not concerned about the dependence of the constants on $\alpha$ here.

\begin{theorem} \label{thm:alpha_large}
Suppose that $h = \cos \alpha$ for $\alpha \in (0, \frac{\pi}{2}]$.
Then there exists an increasing function $\phi \colon \R \to \R$ such that $\phi - \pi$ is odd and
$m = (\cos \phi, \sin \phi)$ is
a minimiser of $E_h$ in $\A_h(1 - \alpha/\pi)$. Furthermore,
the function $f = \cos \alpha - \cos \phi$ satisfies
\[
\limsup_{x_1 \to \pm\infty} \left(x_1^2 |f(x_1)| + |x_1|^3 |f'(x_1)| + x_1^4 |f''(x_1)| + x_1^2 |\Lambda f(x_1)|\right)< \infty.
\]
\end{theorem}

\begin{proof}
This can be proved with the same arguments.

\end{proof}

\section{Concentration compactness} \label{sect:concentration-compactness}

\subsection{Strategy}

We want to prove Theorem \ref{thm:existence_h>1} and Theorem \ref{thm:existence_h<1}
through the analysis of minimising sequences for $E_h$ in the sets
$\A_h(d)$. Similarly to many other variational problems involving
topological information, the main difficulty in proving existence of
minimisers is a possible `escape to infinity' of a topologically
non-trivial part of the members of a minimising sequence.
(This corresponds to the `dichotomy' case in the concentration-compactness
framework of Lions \cite{Lions:84}.) In order
to prevent this, we want to improve Proposition \ref{prop:subadditivity} by showing that
\begin{equation} \label{eqn:strict_subadditivity}
\E_h(d) < \E_h(d_1) + \E_h(d_2)
\end{equation}
for all appropriate decompositions $d = d_1 + d_2$ into smaller
winding numbers. We will achieve this by constructing a magnetisation profile
of winding number $d$ from two energy minimisers in $\A_h(d_1)$ and $\A_h(d_2)$
and estimating the energy {(see Theorems \ref{thm:strict_subadditivity_h>1} and \ref{thm:strict_subadditivity_h<1} below)}. This is where the analysis of
the Euler-Lagrange equation from the previous sections,
and in particular the decay at $\pm \infty$, will be crucial.

In this chapter, we show how inequalities of the type \eqref{eqn:strict_subadditivity}
give rise to minimisers in $\A_h(d)$. Due to the symmetry proved
in Lemma \ref{lem:symmetry}, we may in fact work with a somewhat
weaker hypothesis than expected.

\subsection{Statement}

We formulate the following result for $d \in \N - \frac{\alpha}{\pi}$
only (which in the case $h > 1$ means $d \in \N$). Although a similar
statement would always be true, we do not expect that the hypothesis of Theorem \ref{thm:existence}
will be satisfied if $h < 1$ and $d \in \N$ or $d \in \N + \frac{\alpha}{\pi}$.
Of course we automatically obtain statements for
$d \in \frac{\alpha}{\pi} - \N$ as well.

\begin{theorem}[Concentration compactness] \label{thm:existence}
Suppose that $d = \ell - \alpha/\pi$ for some $\ell \in \N$ such
that
\[
\E_h(d) < 2\E_h(d') + \E_h(d - 2d')
\]
for $d' = 1 - \alpha/\pi, 1, 2 - \alpha/\pi, 2, \ldots, \ell/2 - 1, \ell/2 - \alpha/\pi$
if $\ell$ is even and for
$d' = 1 - \alpha/\pi, 1, 2 - \alpha/\pi, 2, \ldots, (\ell - 1)/2 -\alpha/\pi, (\ell - 1)/2$
if $\ell$ is odd. Then $E_h$ attains its infimum in $\A_h(d)$.
\end{theorem}

\begin{proof} We divide the proof in several steps.

\paragraph{Step 1: pick a minimising sequence.}
Consider a minimising sequence $(m^j)_{j \in \N}$ of $E_h$
in $\A_h(d)$. By Lemma \ref{lem:symmetry}, we may assume that
each $m^j$ is symmetric. In particular, we have $m^j(0) = ((-1)^\ell, 0)$
for every $j \in \N$.
It is clear that a subsequence converges weakly in $H^{1}_\loc(\R; \Ss^1)$.
We may assume without loss of generality that this applies to
the whole sequence, i.e., that $m^j \rightharpoonup m$
weakly in $H^{1}_\loc(\R; \Ss^1)$ for some $m \in H^{1}_\loc(\R; \Ss^1)$.
Then $m$ is symmetric as well with $m(0)=((-1)^\ell, 0)$.
It is also clear that the energy is lower semicontinuous with
respect to such convergence. Thus
\[
E_h(m) \le \liminf_{j \to \infty} E_h(m^j) = \E_h(d).
\]
In particular, we have $\lim_{x_1 \to \pm \infty} m_1(x_1) = k$,
and the winding number $\tilde{d} = \deg(m)$ is well-defined and belongs to $\Z + \{0, \pm \alpha/\pi\}$.
Because of the symmetry and because $m(0) = (\pm 1, 0)$,
we have $\tilde{d} \not= 0$.
If we can show that $\tilde{d} = d$, then it follows that
$m \in \A_h(d)$ and that $m$ is a minimiser of $E_h$ in
this set, which then concludes the proof. The aim of the next steps is to show that $\tilde{d} = d$. 

\paragraph{Step 2: some properties of the minimising sequence.}
First note that
in the case $h < 1$, we obviously have $(m_1 - h)^2 \le 2W(m)$,
whereas in the case $h > 1$, we have
\[
(m_1 - 1)^2  \le 2(1 - m_1) \le 2(1 - m_1) + \frac{1}{h - 1} (1 - m_1)^2 = \frac{2}{h - 1} W(m).
\]
Hence $m_1 - k \in L^2(\R)$, which implies that
\[
\lim_{j \to \infty} \int_{[-2j, -j] \cup [j, 2j]} (m_1 - k)^2 \, dx_1 = 0.
\]
Without loss of generality, we may assume that
\begin{equation} \label{eqn:L^2-close}
\int_{-2j}^{2j} |m^j - m|^2 \, dx_1 \le \frac{1}{j^5}
\end{equation}
for every $j \in \N$ (as we can always select a subsequence with this property and then relabel the indices). Then
\begin{equation} \label{eqn:L^2-convergence}
\lim_{j \to \infty} \int_{[-2j, -j] \cup [j, 2j]} (m_1^j - k)^2 \, dx_1 = 0
\end{equation}
as well. Since
\begin{equation} \label{eqn:H^1-bound}
\limsup_{j \to \infty} \int_{-\infty}^\infty |(m^j)'|^2 \, dx_1 < \infty
\end{equation}
and
\[
\left\|\frac{d}{dx_1} |m - m^j|^2\right\|_{L^1(-2j, 2j)} \le 2\|m' - (m^j)'\|_{L^2(\R)} \|m - m^j\|_{L^2(-2j, 2j)},
\]
then \eqref{eqn:L^2-close} and \eqref{eqn:H^1-bound}, together with
the fact that $m(0)=m^j(0)$, yield:
\begin{equation} \label{eqn:uniform_convergence}
\limsup_{j \to \infty} j \|m^j - m\|_{L^\infty(-2j, 2j)} = 0.
\end{equation}
Similarly, as there exist $t_j \in (j, 2j)$ and $s_j \in (-2j, -j)$
such that $m_1^j(t_j), m_1^j(s_j)\to k$ as $j\to \infty$, we deduce
\begin{equation} \label{eqn:L^infty_estimate}
\lim_{j \to \infty} \|m_1^j - k\|_{L^\infty([-2j, -j] \cup [j, 2j])} = 0.
\end{equation}
Moreover, it follows from
\eqref{eqn:uniform_convergence} that
\begin{equation} \label{eqn:degree}
\lim_{j \to \infty} \int_{-2j}^{2j} (m^j)^\perp \cdot (m^j)' \, dx_1 = 2\pi \tilde{d}.
\end{equation}
Indeed, let $\phi$ and $\phi^j$ be continuous liftings of $m$ and $m^j$, respectively. Due to \eqref{eqn:uniform_convergence}, 
we may assume that
$\|\phi^j - \phi\|_{L^\infty(-2j, 2j)} \to 0$, too. As
$$\int_{-2j}^{2j} (m^j)^\perp \cdot (m^j)' \, dx_1=\phi^j(2j)-\phi^j(-2j),$$ 
we conclude that \eqref{eqn:degree} holds true.

\paragraph{Step 3: cut-off.}
Choose $\eta \in C^\infty(\R)$ with $\eta \equiv 0$
in $(-\infty, 0]$, $\eta \equiv 1$ in $[1, \infty)$, and $0 < \eta < 1$
in $(0, 1)$. For $j \in \Z \setminus \{0\}$, let
$\hat{\eta}_j(x_1) = \eta(4x_1/j - 7)$ and $\tilde{\eta}_{|j|}(x_1) = \eta(4x_1/j - 4)+\eta(-4x_1/j - 4)$. Now define, for $j \in \Z \setminus \{0\}$, the functions
\[
\hat{m}_1^j = \hat{\eta}_j m_1^{|j|} + (1 - \hat{\eta}_j) k
\]
(cut off to the left of $2j$ if $j > 0$ and to the right
of $-2j$ if $j < 0$) and
\[
\tilde{m}_1^{j} = (1 - \tilde{\eta}_{j}) m_1^{j} + \tilde{\eta}_{j}  k, \quad j>0
\]
(cut off outside of $(-j, j)$). {Note that for $j\in \N$, the functions $\tilde{m}_1^{j}-k$, $\hat{m}_1^{-j}-k$ and $\hat{m}_1^{j}-k$ have disjoint support.}
For $j \in \Z$ with $|j|$ sufficiently large, owing to
\eqref{eqn:L^infty_estimate}, there exist
functions $\hat{m}_2^j \colon \R \to [-1, 1]$
such that $\hat{m}_2^j(x_1) = m_2^{|j|}(x_1)$ if $j > 0$ and $x_1 \ge 2j$ or
$j < 0$ and $x_1 \le -2j$ and such that
$\hat{m}^j = (\hat{m}_1^j, \hat{m}_2^j)$ takes values in $\Ss^1$.
Similarly, for $j \in \N$ sufficiently large, there
exists a function $\tilde{m}_2^j \colon \R \to [-1, 1]$
such that $\tilde{m}_2^j(x_1) = m_2^j(x_1)$ for $|x_1| \le j$
and such that $\tilde{m}^j = (\tilde{m}_1^j, \tilde{m}_2^j)$
takes values in $\Ss^1$. The aim of the next steps is to prove that
\begin{equation} \label{eqn:limiting_energy}
\limsup_{j \to \infty} \left(E_h(\tilde{m}^j) + E_h(\hat{m}^j) + E_h(\hat{m}^{-j})\right) \le \E_h(d).
\end{equation}

\paragraph{Step 4: estimate the anisotropy and exchange energy.}
Because we have the pointwise inequalities
$W(m^j) \ge W(\hat{m}^j)$ and $W(m^j) \ge W(\tilde{m}^j)$, it is clear
that
\[
\limsup_{j \to \infty} \int_{-\infty}^\infty \left(W(\tilde{m}^j) + W(\hat{m}^j) + W(\hat{m}^{-j}) - W(m^j)\right) \, dx_1 \le 0.
\]
In order to estimate the exchange energy, note first that in $[-2j, -j]\cup [j, 2j]$, we have
$$
\left((\tilde{m}_1^j)'\right)^2 = (1 - \tilde{\eta}_j)^2 \left((m_1^j)'\right)^2 - 2(1 - \tilde{\eta}_j)\tilde{\eta}_j'  (m_1^j - k)(m_1^j)' + (\tilde{\eta}_j')^2(m_1^j - k)^2
$$
and
\[
\left((\hat{m}_1^{\pm j})'\right)^2 = \hat{\eta}_{\pm j}^2 \left((m_1^j)'\right)^2 + 2\hat{\eta}_{\pm j} \hat{\eta}_{\pm j}' (m_1^j - k) (m_1^j)' + (\hat{\eta}_{\pm j}')^2(m_1^j - k)^2.
\]

In the case $h < 1$, the integrals of the last two terms in each
identity over $(-2j, -j)\cup (j,2j)$ will tend to $0$ as $j \to \infty$
due to \eqref{eqn:L^2-convergence}, \eqref{eqn:H^1-bound}, and
the inequalities
$\|(\hat{\eta}^{\pm j})'\|_{L^\infty(\R)} + \|(\tilde{\eta}^j)'\|_{L^\infty(\R)} \leq \|\eta'\|_{L^\infty(\R)}/j$.
Because of \eqref{eqn:L^infty_estimate}, we have
\[
1 - (\tilde{m}_1^j)^2 \to \sin^2 \alpha \quad \text{and} \quad 1 - (\hat{m}_1^{\pm j})^2 \to \sin^2 \alpha
\]
uniformly in $[-2j, -j]\cup [j, 2j]$, as well as $1 - (m_1^j)^2 \to \sin^2 \alpha$.
It follows that
\begin{equation} \label{eqn:no_energy_gain}
\limsup_{j \to \infty} \int_{-\infty}^{\infty} \left(\frac{\left((\tilde{m}_1^j)'\right)^2}{1 - (\tilde{m}_1^j)^2} + 
\frac{\left((\hat{m}_1^j)'\right)^2}{1 - (\hat{m}_1^j)^2}+ \frac{\left((\hat{m}_1^{-j})'\right)^2}{1 - (\hat{m}_1^{-j})^2} - \frac{\left((m_1^j)'\right)^2}{1 - (m_1^j)^2}\right) \, dx_1 \le 0.
\end{equation}

In the case $h > 1$, we note that
\[
1 - \tilde{m}_1^j = (1 - \tilde{\eta}_j)(1 - m_1^j) \quad \text{and} \quad 1 - \hat{m}_1^j = \hat{\eta}_j(1 - m_1^j)
\]
for $j>0$.
Due to the uniform convergence of $1 + \tilde{m}_1^j \to 2$, $1 + \hat{m}_1^j \to 2$, as well as $1 + m_1^j \to 2$ in $[j, 2j]$ as $j\to \infty$, 
estimating the exchange energy reduces to analysing the following terms:
\[
\frac{\left((\tilde{m}_1^j)'\right)^2}{1 - \tilde{m}_1^j} = (1 - \tilde{\eta}_j) \frac{\left((m_1^j)'\right)^2}{1 - m_1^j} + 2\tilde{\eta}_j' (m_1^j)' - \frac{(\tilde{\eta}_j')^2}{1 - \tilde{\eta}_j} (m_1^j - 1)
\]
and
\[
\frac{\left((\hat{m}_1^j)'\right)^2}{1 - \hat{m}_1^j} = \hat{\eta}_j \frac{\left((m_1^j)'\right)^2}{1 - m_1^j} - 2\hat{\eta}_j' (m_1^j)' - \frac{(\hat{\eta}_j')^2}{\hat{\eta}_j} (m_1^j - 1).
\]
By l'H\^opital's rule,
\[
\lim_{x_1 \nearrow 1} \frac{(\eta'(x_1))^2}{1 - \eta(x_1)} = - 2\eta''(1) = 0,
\]
and thus the function $\frac{(\tilde{\eta}_j')^2}{1 - \tilde{\eta}_j}$ is bounded and supported on $[-2j, -j]\cup [j, 2j]$. Similar arguments apply to $\hat{\eta}_j$.
Moreover, we obviously have
\[
\|\tilde{\eta}_j' (m_1^j)'\|_{L^1(\R)} \le \frac{4}{\sqrt{j}} \|\eta'\|_{L^2(\R)} \|(m^j)'\|_{L^2(\R)} \to 0
\]
as $j \to \infty$. Since the corresponding
estimates hold in $[-2j, -j]$ for $\hat{m}_1^{-j}$
instead of $\hat{m}_1^j$, we conclude that \eqref{eqn:no_energy_gain} holds true in the case $h > 1$, too.

\paragraph{Step 5: estimate the stray field energy.}
Next we want to estimate $\|\tilde{m}_1^j - k\|_{\dot{H}^{1/2}(\R)}$ and
$\|\hat{m}_1^j - k\|_{\dot{H}^{1/2}(\R)}$. Let
$\tilde{v}_j = V(\tilde{m}^j)$ and $\hat{v}_{\pm j} = V(\hat{m}^{\pm j})$ as defined in \eqref{def:Vm}.
Furthermore, let $v_j = V(m^j)$ and $w_j = v_j - \tilde{v}_j - \hat{v}_j - \hat{v}_{-j}$
for $j \in \N$. Then $w_j(\blank, 0) \to 0$ in $L^2(\R)$ by
\eqref{eqn:L^2-convergence}, while $w_j'(\blank, 0)$ remains bounded
in $L^2(\R)$ by \eqref{eqn:H^1-bound}. Therefore, standard interpolation between $\dot{H}^1(\R)$ and $L^2(\R)$ implies that
\[
\lim_{j \to \infty} \int_{\R_+^2} |\nabla w_j|^2 \, dx =\lim_{j \to \infty} \|w_j(\cdot, 0)\|_{\dot{H}^{1/2}(\R)}^2 = 0.
\]
Hence
\begin{equation} \label{eqn:limit_v_j}
\lim_{j \to \infty} \int_{\R_+^2} (|\nabla \tilde{v}_j + \nabla \hat{v}_j + \nabla \hat{v}_{-j}|^2 - |\nabla v_j|^2) \, dx = 0
\end{equation}
by the triangle inequality.
Moreover, as the sequences $(\tilde{m}_1^j - k)_{j \in \N}$ and $(\hat{m}_1^j - k)_{j \in \Z \setminus \{0\}}$ are bounded in
$H^{1}(\R)$, it follows that
\[
\limsup_{j \to \infty} \int_{\R_+^2} \left(|\nabla \tilde{v}_j|^2 + |\nabla \hat{v}_j|^2 + |\nabla \hat{v}_{-j}|^2 \right)\, dx < \infty.
\]
By Lemma \ref{lem:separated_supports}, integration by parts yields
\[
\lim_{j \to \infty} \int_{\R_+^2} \nabla \tilde{v}_j \cdot \nabla \hat{v}_{\pm j} \, dx \stackrel{\eqref{egalite}}{=}-\lim_{j \to \infty} 
\int_{\R} \Lambda(\tilde{m}_1^j - k) (\hat{m}_1^{\pm j}-k) \, dx_1\stackrel{ \eqref{eqn:H^{1/2}-inner_product}}{=} 0
\]
and
\[
\lim_{j \to \infty} \int_{\R_+^2} \nabla \hat{v}_j \cdot \nabla \hat{v}_{-j} \, dx = 0.
\]
Therefore, in view of \eqref{eqn:limit_v_j}, we obtain
\[
\lim_{j \to \infty} \int_{\R_+^2} \left(|\nabla \tilde{v}_j|^2 + |\nabla \hat{v}_j|^2 + |\nabla \hat{v}_{-j}|^2 - |\nabla v_j|^2\right) \, dx = 0.
\]
Now \eqref{eqn:limiting_energy} is proved.

\paragraph{Step 6: conclusion.}
We conclude from \eqref{eqn:L^infty_estimate} and
\eqref{eqn:degree} that $\deg(\tilde{m}^j) = \tilde{d}$
whenever $j$ is sufficiently large. Then by the symmetry, we have
$\deg(\hat{m}^{\pm j}) = \frac{1}{2} (d - \tilde{d})$.
Because of \eqref{eqn:limiting_energy}, we have
\[
\E_h(|\tilde{d}|) + 2\E_h\left(\frac{1}{2}|d - \tilde{d}|\right) \le \E_h(d).
\]
It is clear that $\E_h(\delta) > 0$ whenever $\delta \not= 0$.
Therefore, we can draw the following conclusions from the
above inequality. First, we have already seen that $\tilde{d} \not= 0$.
Second, we conclude that $\tilde{d} \le d$. (Otherwise, Proposition
\ref{prop:monotonicity} would imply that $\E_h(|\tilde{d}|) \ge \E_h(d)$,
which is inconsistent with the inequality.) Third, we conclude that
$\tilde{d} \ge 0$. (Otherwise, set $d_1 = \frac{1}{2}(d - \tilde{d})$
and choose the largest number $d_2 \le \frac{1}{2}(d - \tilde{d})$
such that Proposition \ref{prop:subadditivity} applies to
$d_1$ and $d_2$. Then $d_1 + d_2 \ge d$ and hence
$\E_h(d) \le \E_h(d_1) + \E_h(d_2) \le 2\E_h(\frac{1}{2}|d - \tilde{d}|)$,
contradicting the inequality again.) So we have $0 < \tilde{d} \le d$.

If $h > 1$ or $\alpha = \frac{\pi}{2}$, it is readily seen that
the hypothesis of the theorem excludes all possibilities except
$\tilde{d} = d$. If $h < 1$ and $\alpha \in (0, \frac{\pi}{2})$,
we note that the assumption $d \in \N - \alpha/\pi$
implies that
\[
\lim_{x_1 \to \pm \infty} m_2^j(x_1) = \mp \sin \alpha
\]
for every $j \in \N$. Due to the construction, $\hat{m}^j$ agrees with $m^j$ in $[2j, \infty)$ for $j > 0$
and in $(-\infty, -2j]$ for $j < 0$, respectively; therefore,
$\lim_{x_1 \to \pm \infty} \hat{m}_2^{\pm j}(x_1) = \mp \sin \alpha$
as well, and it follows that $\frac{1}{2}(d - \tilde{d}) \in \Z + \{0, -\alpha/\pi\}$.
Thus in this case as well, the hypothesis of the theorem
excludes all possibilities except $\tilde{d} = d$.
\end{proof}

\section{Proofs of the main results} \label{sect:proofs}

\subsection{Proof of Theorem \ref{thm:existence_h>1}}

For the proof of Theorem \ref{thm:existence_h>1}, it now
suffices to show that the strict inequalities required for Theorem~\ref{thm:existence} are satisfied
in the relevant situation.

\begin{theorem} \label{thm:strict_subadditivity_h>1}
Suppose that $h > 1$ and $d_1, d_2 \in \N$ are such that $E_h$
attains its infima in $\A_h(d_1)$ and in $\A_h(d_2)$.
Then $\E_h(d_1 + d_2) < \E_h(d_1) + \E_h(d_2)$.
\end{theorem}

\begin{proof}
Let $\epsilon > 0$.
Suppose that $m^1 \in \A_h(d_1)$ and $m^2 \in \A_h(d_2)$
are such that
\[
E_h(m^1) = \E_h(d_1) \quad \text{and} \quad E_h(m^2) = \E_h(d_2).
\]
Then by Theorem \ref{thm:exponential_decay}
and Proposition \ref{prop:localisation},
there exist $a, b, c > 0$ such that for any $R > a$, we can construct
$\tilde{m}^1 \in \A_h(d_1)$ and $\tilde{m}^2 \in \A_h(d_2)$
with $\tilde{m}_1^1 \equiv 1$ and $\tilde{m}_1^2 \equiv 1$
outside of $(-2R, 2R)$ and such that
\[
E_h(\tilde{m}^1) \le \E_h(d_1) + ce^{-bR} \quad \text{and} \quad E_h(m^2) \le \E_h(d_2) + ce^{-bR}.
\]
Since $d_1 \not= 0$ and $d_2 \not= 0$, {by Lemma \ref{lem:deflections_cost_energy},} there exists a universal
constant $C_1 > 0$ such that
\[
\|1 - \tilde{m}_1^1\|_{L^1(\R)} \ge C_1 \quad \text{and} \quad \|1 - \tilde{m}_1^2\|_{L^1(\R)} \ge C_1.
\]
Suppose that $\tilde{m}^1 = (\cos \tilde{\phi}_1, \sin \tilde{\phi}_1)$ and
$\tilde{m}^2 = (\cos \tilde{\phi}_2, \sin \tilde{\phi}_2)$
with $\tilde{\phi}_1(x_1) = 0$ for $x_1 \ge 2R$ and $\tilde{\phi}_2(x_1) = 0$
for $x_1 \le -2R$. Then we define
\[
\phi(x_1) = \begin{cases}
\tilde{\phi}_1(x_1 + 6R) & \text{if $x_1 < 0$}, \\
\tilde{\phi}_2(x_1 - 6R) & \text{if $x_1 \ge 0$}.
\end{cases}
\]
Let $m = (\cos \phi, \sin \phi)$. Then $\deg(m) = d_1 + d_2$, and
we have
\[
\int_{-\infty}^\infty \left(|m'|^2 + 2W(m)\right) \, dx = \int_{-\infty}^\infty \left(|(\tilde{m}^1)'|^2 + |(\tilde{m}^2)'|^2 + 2W(\tilde{m}^1) + 2W(\tilde{m}^2)\right) \, dx.
\]
Let $\tilde{u}_1 = U(\tilde{m}^1)$ and  $\tilde{u}_2 = U(\tilde{m}^2)$ defined as in \eqref{def:Um}.
Furthermore, let $u = U(m)$. Then, by the uniqueness of $U(m)$, we have $u(x)=\tilde{u}_1(x_1 + 6R, x_2)+\tilde{u}_2(x_1 - 6R, x_2)$ for $x\in \R^2_+$.
Then we also have
\[
\int_{\R_+^2} |\nabla u|^2 \, dx = \int_{\R_+^2} \left(|\nabla \tilde{u}_1|^2 + |\nabla \tilde{u}_2|^2\right) \, dx + 2 \int_{\R_+^2} \nabla \tilde{u}_1(x_1 + 6R, x_2) \cdot \nabla \tilde{u}_2(x_1 - 6R, x_2) \, dx.
\]
By Lemma \ref{lem:attraction}, we have
\[
\begin{split}
\int_{\R_+^2} \nabla \tilde{u}_1(x_1 + 6R, x_2) \cdot \nabla \tilde{u}_2(x_1 - 6R, x_2) \, dx &\stackrel{\eqref{eqn:H^{1/2}-inner_product}}{\le} -\frac{1}{256\pi R^2} \|1 - \tilde{m}_1^1\|_{L^1(\R)} \|1 - \tilde{m}_1^2\|_{L^1(\R)} \\
&\le - \frac{C_1^2}{256 \pi R^2}.
\end{split}
\]
Hence
\[
E_h(m) \le E_h(\tilde{m}^1) + E_h(\tilde{m}^2) - \frac{C_1^2}{256\pi R^2} \le \E_h(d_1) + \E_h(d_2) + 2ce^{-bR} - \frac{C_1^2}{256\pi R^2}.
\]
Therefore,
\[
\E_h(d_1 + d_2) \le \E_h(d_1) + \E_h(d_2) + 2ce^{-bR} - \frac{C_1^2}{16 \pi R^2}.
\]
For $R$ sufficiently large, this yields the desired inequality.
\end{proof}

\begin{proof}[Proof of Theorem \ref{thm:existence_h>1}]
It suffices to consider $d \in \N$; indeed, for $d=0$, a constant configuration will minimise
$E_h$ in $\A_h(0)$ and the case $d \in -\N$ is reduced to
$d \in \N$ by a change of orientation.

We prove the statement by induction. For $d = 1$,
it follows from  the symmetrisation arguments of Melcher \cite{Me1} and
Chermisi-Muratov \cite{Chermisi-Muratov:13} that a minimiser exists.
Now suppose that minimisers exist in $\A_h(d')$ for any
$d' = 1, \ldots, d - 1$. Then Theorem \ref{thm:strict_subadditivity_h>1}
implies that
\[
\E_h(d) < \E_h(d') + \E_h(d - d')
\]
for $d' = 1, \ldots, d - 1$. It follows that the hypothesis of Theorem \ref{thm:existence}
is satisfied and that $\E_h$ is attained in $\A_h(d)$.
\end{proof}

\subsection{Proof of Theorem \ref{thm:existence_h<1}}

Similarly to the previous section, the following strict inequality is the key here.

\begin{theorem} \label{thm:strict_subadditivity_h<1}
There exists a number $H \in [0, 1)$ such that whenever $h=\cos \alpha \in [H, 1)$,
\[
\E_h(2 - \alpha/\pi) < 2\E_h(1 - \alpha/\pi) + \E_h(\alpha/\pi).
\]
\end{theorem}

\begin{proof}
Let $m^\sharp \in \A_h(\alpha/\pi)$ be a minimiser as in Theorem
\ref{thm:decay_h<1} and let $m^\flat \in \A_h(1 - \alpha/\pi)$
be a minimiser as in Theorem \ref{thm:alpha_large}. Set
$f^\sharp = m_1^\sharp - h$ and $f^\flat = m_1^\flat - h$.
Then $f^\sharp \ge 0$ and $f^\flat \le 0$. We have
\[
\|f^\sharp\|_{L^1(\R)} \le C_1 \alpha
\]
for a universal constant $C_1$ by Theorem \ref{thm:decay_h<1}.
Furthermore, by the decay established in this theorem, we may
apply Proposition \ref{prop:localisation} to $m^\sharp$
with three functions $\omega, \sigma, \tau$ that satisfy
\[
\omega(x_1) {\leq } \frac{C_2}{x_1^2}, \quad \sigma(x_1) \le \frac{C_2}{|x_1|^3}, \quad \tau(x_1) \le \frac{C_2}{x_1^2},
\]
for $x_1 \ge c/\alpha$, where $c, C_2 > 0$ are universal constants.
Hence for any $R > 2c/\alpha$ there exists a constant $C_3$
(possibly depending on $h$, but on nothing else) such that
there is a map $\tilde{m}^\sharp \in \A_h(\alpha/\pi)$ with
$\tilde{m}_1^\sharp = \cos \alpha$ outside of $[-R, R]$ and
\[
E_h(\tilde{m}^\sharp) \le \E_h(\alpha/\pi) + \frac{C_3}{R^3}.
\]
Furthermore, the function $\tilde{f}^\sharp = \tilde{m}_1^\sharp - h\geq 0$ still
satisfies
\begin{equation} \label{eqn:small_L^1}
\|\tilde{f}^\sharp\|_{L^1(\R)} \le C_1 \alpha.
\end{equation}
Similarly, there exists a map $\tilde{m}^\flat \in \A_h(1 - \alpha/\pi)$
such that $\tilde{m}_1^\flat = h$ outside of $[-R, R]$ and
\[
E_h(\tilde{m}^\flat) \le \E_h(1 - \alpha/\pi) + \frac{C_4}{R^3},
\]
where $C_4 = C_4(h)$.

Since $m^\flat \in \A_h(1 - \alpha/\pi)$ is symmetric,
we have $m^\flat(0) = -1$ (as a complex number on $\Ss^1$).
By Lemma \ref{lem:energy_estimate_for_1-alpha/pi}, there
exists a universal constant $C_5$ such that $E_h(m^\flat) \le C_5$.
Thus we obtain a universal bound for $m_1^\flat - h$ in
$H^{1}(\R)$ and therefore in $C^{0, 1/2}([-1, 1])$.
The same is true for $\tilde{m}^\flat_1 - h$ whenever $\alpha \le c$ (as $R>2$),
because in this case, the two functions agree in $[-1, 1]$.
It follows that for $\tilde{f}^\flat = \tilde{m}_1^\flat - h\leq 0$, we have
\begin{equation} \label{eqn:large_L^1}
\|\tilde{f}^\flat\|_{L^1(\R)} \ge C_6
\end{equation}
for a universal constant $C_6 > 0$.

Define 
\[
m(x_1) = \begin{cases}
\tilde{m}^\flat(x_1 + 4R) & \text{if $x_1 < -2R$}, \\
\tilde{m}^\sharp(x_1) & \text{if $|x_1| \le 2R$}, \\
\tilde{m}^\flat(x_1 - 4R) & \text{if $x_1 > 2R$}.
\end{cases}
\]
Then $m \in \A_h(2 - \alpha/\pi)$ and the arguments from the proof of Theorem \ref{thm:strict_subadditivity_h>1} (used
to compute the stray field) yield:
\begin{multline*}
E_h(m) \le \E_h(\alpha/\pi) + 2\E_h(1 - \alpha/\pi) + \frac{C_3 + 2C_4}{R^3} \\
+ \langle \tilde{f}^\flat(\blank + 4R), \tilde{f}^\flat(\blank - 4R)\rangle_{\dot{H}^{1/2}(\R)} + \langle \tilde{f}^\flat(\blank + 4R), \tilde{f}^\sharp\rangle_{\dot{H}^{1/2}(\R)} + \langle \tilde{f}^\sharp, \tilde{f}^\flat(\blank - 4R)\rangle_{\dot{H}^{1/2}(\R)}.
\end{multline*}
By Lemma \ref{lem:attraction} and \eqref{eqn:large_L^1}, we have
a universal constant $C_7$ such that
\[
\langle \tilde{f}^\flat(\blank + 4R), \tilde{f}^\flat(\blank - 4R)\rangle_{\dot{H}^{1/2}(\R)} \le -\frac{C_7}{R^2}.
\]
On the other hand, using \eqref{eqn:small_L^1}, we obtain
another universal constant $C_8$ with 
\[
\langle \tilde{f}^\flat(\blank + 4R), \tilde{f}^\sharp\rangle_{\dot{H}^{1/2}(\R)} + \langle \tilde{f}^\sharp, \tilde{f}^\flat(\blank - 4R)\rangle_{\dot{H}^{1/2}(\R)} \le \frac{C_8 \alpha}{R^2}.
\]
Hence
\[
E_h(m) \le \E_h(\alpha/\pi) + 2\E_h(1 - \alpha/\pi) + \frac{C_3 + 2C_4}{R^3} + \frac{C_8 \alpha - C_7}{R^2}.
\]
If we choose $\alpha$ small enough (i.e., $H$ sufficiently close to $1$)
and $R$ large enough, then
\[
\E_h(2 - \alpha/\pi) \le E_h(m) < \E_h(\alpha/\pi) + 2\E_h(1 - \alpha/\pi),
\]
as required.
\end{proof}

\begin{proof}[Proof of Theorem \ref{thm:existence_h<1}]
This is now a
direct consequence of Theorem \ref{thm:strict_subadditivity_h<1}
and Theorem \ref{thm:existence}.
\end{proof}

\subsection{Proof of Theorem \ref{thm:non-existence}}

The statement of Theorem \ref{thm:non-existence} is an
immediate consequence of Proposition \ref{prop:subadditivity}
and the following result.

\begin{lemma}
If $h < 1$, then for any $m \in \A_h(1)$,
\[
E_h(m) > \E_h(\alpha/\pi) + \E_h(1 - \alpha/\pi).
\]
\end{lemma}

\begin{proof}
Define $m_1^+ = \max\{m_1, k\}$ and $m_1^- = \min\{m_1, k\}$.
Then clearly there exist $a_+, a_- \in \R$ such that
$m_1^+(a_+) = 1$ and $m_1^-(a_-) = - 1$.
Thus by Lemma \ref{lem:degree}, there exist $m_2^\pm \colon \R \to [-1, 1]$
such that $m^+ = (m_1^+, m_2^+) \in \A_h(\alpha/\pi)$
and $m^- = (m_1^-, m_2^-)$ in $\A_h(1 - \alpha/\pi)$.
Moreover,
\[
\int_{-\infty}^\infty \left(\frac{1}{2} |m'|^2 + W(m)\right) \, dx_1 = \int_{-\infty}^\infty \left(\frac{1}{2} |(m^+)'|^2 + W(m^+) + \frac{1}{2} |(m^-)'|^2 + W(m^-)\right) \, dx_1.
\]
Hence Lemma \ref{lem:repulsion} implies that
$E_h(m) > E_h(m^+) + E_h(m^-) \ge \E_h(1 - \alpha/\pi) + \E_h(\alpha/\pi)$.
\end{proof}

\section*{Appendix. Nonexistence of critical points in a local model}

In order to highlight the role of the nonlocal term for the existence of
minimisers (or even critical points) carrying a winding number $d\geq 1$
for our variational problem, we discuss the corresponding model
without the nonlocal term. In this situation, we have a well-known
nonexistence result.

For $h\geq 0$, $h\neq 1$, we consider the following
Allen-Cahn type energy defined for $\phi:\R\to \R$ (representing the angle of an $\Ss^1$-valued transition layer $m=(\cos \phi, \sin \phi)$):
$$F_h(\phi)=\frac{1}{2} \int_{-\infty}^\infty \left((\phi')^2 + 2W(\phi)\right) \, dt.$$
Here we use the same potential $W$ as in \eqref{aniso}. That is,
$$
W(\phi) = \begin{cases}
\frac 1 2 (\cos \phi-h)^2 & \text{if } h<1,\\
\frac 1 2 (2h-1-\cos \phi)(1-\cos \phi) & \text{if } h>1.
\end{cases} 
$$
The Euler-Lagrange equation associated to a critical point $\phi$ of $F_h$ is now given by
\be
\label{Allen}
\phi''=W'(\phi)\quad \text{in } \R.\ee
Denote again $$\alpha=\arccos \min\{h,1\} \in [0, \frac \pi 2].$$ We impose the following boundary condition at infinity:
$$\phi(\pm \infty):=\lim_{t \to \pm \infty} \phi(t) \in 2\pi\Z+\{-\alpha, \alpha\}.$$
Then the winding number of $m=(\cos \phi, \sin \phi):\R\to \Ss^1$ is given by
$$\deg(m)=\frac{\phi(+\infty)-\phi(-\infty)}{2\pi}\in \Z + \left\{0, \pm \frac{\alpha}{\pi}\right\}.$$

We have the following nonexistence result. Although this is a
well-known fact, we give a proof for completeness.

\begin{theorem}
Suppose that $\phi:\R\to \R$ is a non-constant
solution of equation \eqref{Allen} with boundary condition $\phi(\pm \infty) \in 2\pi\Z+\{-\alpha, \alpha\}$.
Let $d$ be the winding number corresponding to $\phi$.
If $h > 1$, then one has $d = \pm 1$. If $h < 1$, then
one has $d = \pm \alpha/\pi$ or $d = \pm (1 - \alpha/\pi)$.
\end{theorem}

\begin{proof}
First, note that every solution $\phi$ of \eqref{Allen} satisfies
$$(\phi')^2 - 2W(\phi)=q \quad \text{in } \R$$
for some constant $q\in \R$. We want to prove that $q=0$. Indeed, as $\phi$ has finite limits at infinity, the above equation implies that $\phi'(\pm \infty)=\ell_\pm$ for some
$\ell_\pm\in \R$. It is enough to prove that $\ell_\pm=0$. For
this purpose, consider $X=(\phi, \phi')$ and note that $X$
solves the following system of ODEs,
\be
\label{dynam}
X'=V(X),
\ee
generated by the vector field $V(X)=(X_2, W'(X_1))$. Since $t\mapsto X(t)$ stays confined in a compact set of $\R^2$ and has a limit point as $t\to \pm \infty$ 
(by our boundary conditions for the solution $\phi$), this limit point is a critical point of the vector field $V$, i.e., we have
$X'=(0,0)$. This implies that $\ell_\pm=0$, and thus, that $q=0$.
In particular, the trajectory $\{X(t)=(\phi(t), \phi'(t))\}_{t\in \R}$ is included in the zero set of the Hamiltonian
$$H(X_1, X_2)=\frac 1 2 X_2^2-W(X_1), \quad X\in \R^2.$$
We denote $Z^\pm=\set{(X_1, X_2)}{\pm X_2>0, \, H(X_1, X_2)=0}$ and $$Z^0=\set{(X_1, 0)}{H(X_1, 0)=0}=2\pi \Z+\{\pm \alpha\}.$$
It is readily seen that
any connected component of $Z^+$ and $Z^-$ ends at two consecutive points
of $Z^0$ (see Fig.\ \ref{fig:hamiltonian}).
Obviously, any zero of $Z^0$ is a stationary solution of \eqref{dynam}. Therefore, by the uniqueness of solutions to initial value
problems for \eqref{dynam}, the trajectory $\{X(t)=(\phi(t), \phi'(t))\}_t$ begins and ends at two consecutive points in $Z^0$.
That is, in the case $h > 1$, we have winding number $\pm 1$,
and in the case $h < 1$, we have 
$d = \pm \alpha/\pi$ or $d = \pm (1 - \alpha/\pi)$.

\begin{figure}[htbp]
 \centering
 \begin{minipage}{0.4\linewidth}
   \centering
\includegraphics[width=\textwidth]{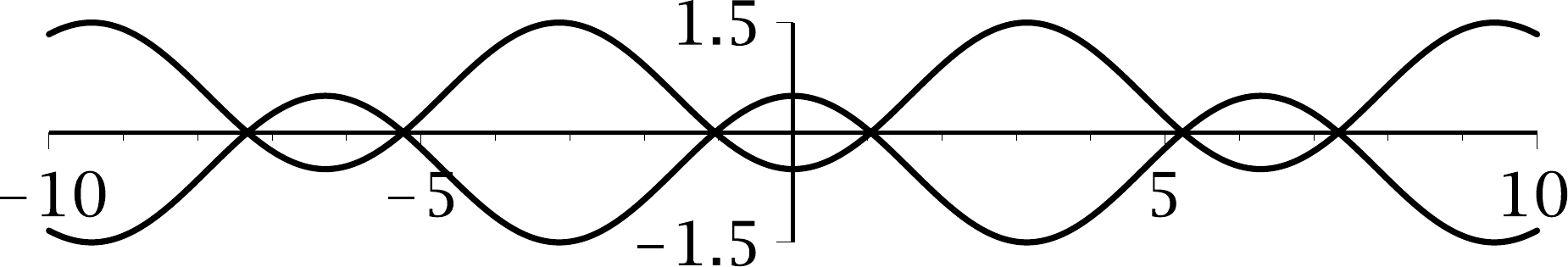}
 \end{minipage}
  \hspace{1cm}
 \begin{minipage}{0.4\linewidth}
  \centering
 \includegraphics[width=\textwidth]{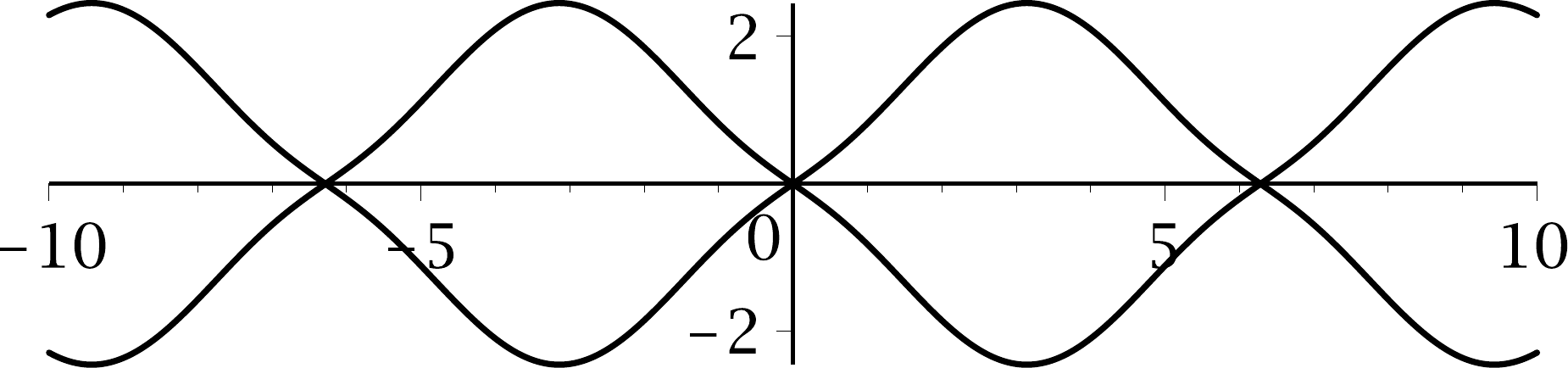}
 \end{minipage}
\caption{The zero set of the Hamiltonian $H$ for $h=\frac 1 2$ (left) and $h=2$ (right). }
  \label{fig:hamiltonian}
\end{figure}

\end{proof}

\bibliographystyle{amsplain}
\bibliography{bib}

\providecommand{\bysame}{\leavevmode\hbox to3em{\hrulefill}\thinspace}
\providecommand{\MR}{\relax\ifhmode\unskip\space\fi MR }
\providecommand{\MRhref}[2]{%
  \href{http://www.ams.org/mathscinet-getitem?mr=#1}{#2}
}
\providecommand{\href}[2]{#2}
\begin{thebibliography}{10}

\bibitem{CMO07}
Antonio Capella, Christof Melcher, and Felix Otto, \emph{Wave-type dynamics in
  ferromagnetic thin films and the motion of {N}\'eel walls}, Nonlinearity
  \textbf{20} (2007), no.~11, 2519--2537.

\bibitem{Chermisi-Muratov:13}
Milena Chermisi and Cyrill~B. Muratov, \emph{One-dimensional {N}\'eel walls
  under applied external fields}, Nonlinearity \textbf{26} (2013), no.~11,
  2935--2950.

\bibitem{Cote-Ignat-Miot}
Rapha{\"e}l C{\^o}te, Radu Ignat, and Evelyne Miot, \emph{A thin-film limit in
  the {L}andau--{L}ifshitz--{G}ilbert equation relevant for the formation of
  {N}\'eel walls}, J. Fixed Point Theory Appl. \textbf{15} (2014), no.~1,
  241--272.

\bibitem{DKO}
Antonio DeSimone, Hans Kn{\"u}pfer, and Felix Otto, \emph{2-d stability of the
  {N}\'eel wall}, Calc. Var. Partial Differential Equations \textbf{27} (2006),
  no.~2, 233--253.

\bibitem{DKMOreduced}
Antonio DeSimone, Robert~V. Kohn, Stefan M{\"u}ller, and Felix Otto, \emph{A
  reduced theory for thin-film micromagnetics}, Comm. Pure Appl. Math.
  \textbf{55} (2002), no.~11, 1408--1460.

\bibitem{DKMO_rep}
\bysame, \emph{Repulsive interaction of {N}\'eel walls, and the internal length
  scale of the cross-tie wall}, Multiscale Model. Simul. \textbf{1} (2003),
  no.~1, 57--104.

\bibitem{DiNezza-Palatucci-Valdinoci:12}
Eleonora Di~Nezza, Giampiero Palatucci, and Enrico Valdinoci,
  \emph{Hitchhiker's guide to the fractional {S}obolev spaces}, Bull. Sci.
  Math. \textbf{136} (2012), no.~5, 521--573.

\bibitem{HS98}
Alex Hubert and Rudolf Sch{\"a}fer, \emph{Magnetic domains: The analysis of
  magnetic microstructures}, Springer-Verlag, Berlin, 1998.

\bibitem{Ig}
Radu Ignat, \emph{A {$\Gamma$}-convergence result for {N}\'eel walls in
  micromagnetics}, Calc. Var. Partial Differential Equations \textbf{36}
  (2009), no.~2, 285--316.

\bibitem{Ignat_Knupfer}
Radu Ignat and Hans Kn\"upfer, \emph{Vortex energy and $360^\circ$ {N}\'eel
  walls in thin-film micromagnetics}, Comm. Pure Appl. Math. \textbf{63}
  (2010), no.~12, 1677--1724.

\bibitem{IgnMosARMA}
Radu Ignat and Roger Moser, \emph{Interaction {E}nergy of {D}omain {W}alls in a
  {N}onlocal {G}inzburg--{L}andau {T}ype {M}odel from {M}icromagnetics}, Arch.
  Ration. Mech. Anal. \textbf{221} (2016), no.~1, 419--485.

\bibitem{IO}
Radu Ignat and Felix Otto, \emph{A compactness result in thin-film
  micromagnetics and the optimality of the {N}\'eel wall}, J. Eur. Math. Soc.
  (JEMS) \textbf{10} (2008), no.~4, 909--956.

\bibitem{Lieb-Loss:01}
Elliott~H. Lieb and Michael Loss, \emph{Analysis}, second ed., Graduate Studies
  in Mathematics, vol.~14, American Mathematical Society, Providence, RI, 2001.

\bibitem{Lions:84}
P.-L. Lions, \emph{The concentration-compactness principle in the calculus of
  variations. {T}he locally compact case. {I}}, Ann. Inst. H. Poincar\'e Anal.
  Non Lin\'eaire \textbf{1} (1984), no.~2, 109--145.

\bibitem{Me1}
Christof Melcher, \emph{The logarithmic tail of {N}\'eel walls}, Arch. Ration.
  Mech. Anal. \textbf{168} (2003), no.~2, 83--113.

\bibitem{Me2}
\bysame, \emph{Logarithmic lower bounds for {N}\'eel walls}, Calc. Var. Partial
  Differential Equations \textbf{21} (2004), no.~2, 209--219.

\end{thebibliography}

\end{document}